\documentclass{article}

\usepackage{rhee}
\usepackage{mathtools}
\usepackage{bbm}
\usepackage[font=small, labelfont=bf]{caption}
\usepackage{graphicx}    
\usepackage{subcaption}  

\DeclarePairedDelimiter\floor{\lfloor}{\rfloor}

\numberwithin{equation}{section}

\begin{document}
\title{Unbiased Derivative Estimation for Stationary Mean of Parameterized Markov chains}
\author{Jeffrey Wang, Chang-Han Rhee}
\date{November 11, 2024}
\maketitle


\begin{center}
\emph{(This is a preliminary version. Comments are welcome. The notation and exposition will be refined in a forthcoming update.)}
\end{center}

\begin{abstract}
\noindent
We propose a new approach to unbiased estimation of the gradients of the stationary means associated with parametrized families of Markov chains. 
Our estimators are particularly efficient when the Markov chains have slow mixing rate. Our approach does not require a specific parametrization except for an oracle to evaluate the transition density and its gradient at a given data point without any additional knowledge about the density function itself. 
It makes our estimator suitable for parametrizations associated with neural networks. 
The estimator can potentially achieve large improvement in terms of efficiency. 
Numerical experiments confirm the good performance predicted by the theory.
\end{abstract}

\section{Introduction \jw{Ready to be reviewed}}
Consider a Markov chain $\newnota{chain-X}{X} = \{X_i\}_{i\geq 0}$ on a general state space \newnota{state-space}{$\mathcal{X}$}. 
Let $\newnota{P-theta}{P(\theta)}:= \left(P(\theta, x, dy): x, y \in \mathcal{X}\right)$ denote its transition kernel parametrized by $\theta \in \newnota{Lam}{\Theta}$. We assume that there exists a unique invariant distribution $\newnota{pi-theta}{\pi(\theta)}$ associated with $P(\theta)$. 
Now let $\newnota{gamma}{\gamma(\theta)} := \mathbb{E}_{\pi(\theta)}[f(X)]$ denote the stationary mean for some integrable function $\newnota{f}{f}: \mathcal{X} \rightarrow \mathbb{R}$ under $\pi(\theta)$. 
The goal of our work is to estimate the derivative of the stationary mean, denoted by
\begin{equation}\label{eq: gamma-prime}
    \newnota{gamma-prime}{\gamma'(\theta)} := \frac{d}{d\theta}\mathbb{E}_{\pi(\theta)}[f(X)].
\end{equation}
Such derivatives are essential for sensitivity analysis of the long-run average performance of a system or for optimization of the performance with respect to the parameter $\theta$. 
However, the explicit evaluation of derivatives are rarely possible and one typically needs to estimate the derivatives computationally. 
When efficient and unbiased derivative estimators are available, numerical optimization algorithms such as stochastic gradient decent (SGD) converge more quickly.
Throughout the paper, efficiency of an estimator is measured by the reciprocal of the product of the expected work and variance of the estimator; see \cite{glynn1992asymptotic}.

Our approach builds on a few recent lines of research.
The first line concerns the general theory \cite{rhee2015unbiased} of \emph{debiasing} simulation estimators. 
In \cite{glynn2014exact}, it was shown that this idea, together with a coupling technique, can be used to construct unbiased estimators for the stationary means of positive Harris recurrent chains. 
\cite{jacob2020unbiased} further studied and exploits the coupling technique to bring it to the Markov chain Monte Carlo (MCMC) setting where they construct unbiased estimators for integrals with respect to a target distribution. 
Recently \cite{douc2022solving} also applied similar coupling idea to construct unbiased estimators for the solutions to the Poisson's equation associated with Markov chains, 
and from there they develop unbiased estimator for the asymptotic variance of the Markov chain. 
In this work, we address an estimation problem of yet another limiting quantity, the derivative of the stationary mean.

There have been a substantial body of works that aim to estimate such a derivative. 
Whenever an atom is available or can be easily constructed, we may utilize the regeneration structure of the Markov chain to construct a derivative estimator,  see \cite{glynn1995likelihood}, for example. 
However, since the estimator is a non-linear function of sample averages, it is an asymptotically consistent estimator but biased for any finite realizations. 
More importantly, this approach can be hard to implement if regeneration states are hard to be identified or the constructed regeneration cycle is too long to be efficient. The Infinitesimal Perturbation Analysis (IPA) estimator for $\gamma'(\theta)$ has also been developed in \cite{glasserman1993regenerative}. 
It transforms the derivative estimation problem into a stationary mean estimation problem by constructing the \emph{derivative process}. 
Although the estimator is unbiased and the efficiency typically triumphs other estimators when available, the method can be implemented on a very restricted class of parametrizations. 
\cite{heidergott2006measure},\cite{heidergott2010perturbation} provide another way to estimate $\gamma'(\theta)$ by applying the Hahn-Jordan decomposition to decompose the derivative of transition kernel $P'(\theta)$ into two positive kernels $\newnota{P+}{P^{+}(\theta)}$ and $\newnota{P-}{P^{-}(\theta)}$. 
They simulate two chains separately starting from those two kernels and produce an estimator, which they call Phantom estimator. 
Potentially, this could avoid requiring constructing atoms and be applied to a broader class of parametrizations but it still needs detailed knowledge about the transition kernel $P(\theta)$ to (i) perform the Hahn-Jordan decomposition, (ii) sample from the decomposed kernels and (iii) compute the normalizing constants for the decomposed parts. 
A more recent work \cite{glynn2019likelihood} proposed two more likelihood ratio estimators for $\gamma'(\theta)$ that does not require constructing regenerative cycles. 
Moreover, this approach requires the knowledge of the the stationary mean $\gamma(\theta)$ which is often unavailable. 
Those likelihood ratio estimators are asymptotically consistent and so far have the fewest requirements to implement. 
However, as the bias gets smaller, the efficiency gets worse and a difficult trade-off has to be made here. 
Specific implementation details and numerical results of the aforementioned methods are deferred to Section \ref{sec-NU}.

Our contribution builds on the theoretical results from \cite{rhee2017lyapunov} and we propose an unbiased estimator for $\gamma'(\theta)$. 
The requirement to implement this estimator is modest. 
To be specific, it only requires an oracle which given a parameter $\theta$ and a state pair $(x, y)$, outputs the one-step transition density $P(\theta, x, dy)$ and its derivative $P'(\theta, x, dy)$ with respect to $\theta$. 
Our method does not require constructing regeneration time nor analytical knowledge of the transition kernel $P(\theta)$. 
This is desirable in modern optimization setting as often the kernel is parametrized by complicated function approximators such as neural networks (NNs), for which it is almost impossible to obtain analytic knowledge. 
Meanwhile, an oracle described above can be obtained if the transition kernel is parameterized by NNs. 
Specifically, given the current NN weights $\theta$ and a state pair $(x, y)$, $P(\theta, x, dy)$ and $P'(\theta, x, dy)$ can be numerically computed using forward-feeding and back-propagation of the NNs respectively. Section \ref{sec-application} provides examples regarding Markov chains whose transition kernels are parametrized by NNs. In addition to the previously stated theoretical desirability, our estimator can be surprisingly efficient as well when tuned appropriately. Numerical experiments show that for slow mixing Markov chains which are the difficult scenarios, the efficiency of our derivative estimator achieves a similar level of performance with the IPA estimator and surpasses almost all previous estimators.

The rest of the paper is organized as follows: Section \ref{sec-prelim} contains the preliminary background for coupling technique and how it can be used to construct unbiased estimators for the stationary means and solutions to the Poisson's equation. Section \ref{sec-UDE} contains the development of our unbiased derivative estimators. Section \ref{sec-EA} provides efficiency analysis of the previously developed estimator and gives a general guideline for choosing parameters based on the analysis results. Section \ref{sec-NU} shows the empirical performance of our estimator comparing with the previously developed ones in a queueing example.

\section{Preliminaries \jw{Ready to be reviewed}}\label{sec-prelim}
\subsection{Construction of coupled Markov chains}
\chr{add line-breaks between every sentence}\jw{done}
Assume we have a joint kernel $\newnota{P-bar}{\Bar{P}(\theta)}$ that evolves on the product space $\mathcal{X} \times \mathcal{X}$ such that the marginal transition probabilities follow $P(\theta)$, i.e. $\Bar{P}(\theta, (x, y), \mathcal{A} \times \mathcal{X}) = P(\theta, x, \mathcal{A})$ and $\Bar{P}(\theta, (x, y), \mathcal{X} \times \mathcal{A}) = P(\theta, y, \mathcal{A})$.
Furthermore, $\Bar{P}(\theta)$ satisfies the property that the marginal chains stay together after they meet. 
That is, $\bar P(\theta, (x,x), \cdot)$ is supported on the diagonal $\{(y,z)\in \mathcal X \times \mathcal X: y = z\}$ for any $x \in \mathcal X$.
Let $X_0$, $Y_0 \sim \mu_0$ where $\newnota{mu-0}{\mu_0}$ is a distribution on $\mathcal{X}$.
For $\newnota{L}{L} \geq 1$, we construct an $L$-lagged pair of coupled Markov chain $\newnota{couple-xy}{\left(X, Y\right)}:= \{X_i, Y_i\}_{i\geq 0}$. 
We first run chain $X$ for $L$ steps where we generate $X_i \sim P(\theta, X_{i-1}, \cdot)$ for $i \leq L$. 
We then run the joint chain and generate future steps according to the joint kernel $\Bar{P}(\theta)$ such that for $i \geq 1$, $\left(X_{i+L}, Y_{i}\right) \sim \Bar{P}\left(\theta, (X_{i+L-1}, Y_{i-1}), \cdot\right)$. Let $\newnota{tau^L}{\tau^L} := \inf\left\{i \geq L: X_{i} = Y_{i-L}\right\}$ and note that $f(X_{i})=f(Y_{i-L})$ for all $i \geq \tau^L$.
In quite general contexts, the coupling kernel $\Bar{P}$ can be constructed in various ways, including maximum coupling and common random numbers. See \cite{lindvall2002lectures} for more details.

\subsection{Notations and Assumptions}
First we define some notations. For a function $h: \mathcal{X} \rightarrow \mathbb{R}$, a transition kernel $P$, a state $x \in \mathcal{X}$ and a probability measure $\mu$, define the following notation:
\begin{equation}
    \newnota{Ph}{Ph(x)} := \int_{\mathcal{X}}h(y)P(x, dy);
\end{equation}
\begin{equation}
    \newnota{mu-h}{\mu h} := \int_{\mathcal{X}}h(y)\mu(dy).
\end{equation}
\chr{Are above $x$ and $y$ correct?}\jw{Notation fixed}
Furthermore, for a function $V: \mathcal{X} \rightarrow \mathbb{R}$, we define the $V$-norm of function $h$ as
\begin{equation}
    \|h\|_{V} := \sup_{x \in \mathcal{X}}\frac{|h(x)|}{V(x)}.
\end{equation}

We make the following assumptions throughout the rest of the paper \todo{specify the range}\jw{Range is throughout all the sections}.
\begin{assumption}\label{assumption: r-g-A4}
\chr{Why was this assumption in red font?}
\jw{Fixed}
\chr{How about $\Lambda \to \Theta$? }
\jw{Changed all $\Lambda \to \Theta$}
\chr{absolute continuity doesn't have to be w.r.t.\ $P(\theta_0)$. Instead, it would suffice that $P(\theta)$ are absolutely continuous w.r.t.\ $P(\theta^*)$ for some $\theta^* \in \Theta$.}\jw{Do you want me to change $\theta_0$ to some $\theta^{*} \in \Theta$ for the paper?}
    Assumption A4/A5 in \cite{rhee2017lyapunov}: The family of one-step transition kernels $(P(\theta): \theta \in \Theta)$ is absolutely continuous with respect to $P\left(\theta_0\right)$, in the sense that there exists a density $(p(\theta, x, y): \theta \in \Theta, x, y \in \mathcal{X})$ for which
    $$
    P(\theta, x, d y)=p(\theta, x, y) P\left(\theta_0, x, d y\right)
    $$
    for $x, y \in \mathcal{X}$, and $\theta \in \Theta$. Furthermore, there exists $\epsilon>0$ for which $p(\cdot, x, y)$ is continuously differentiable on $\left[\theta_0-\epsilon, \theta_0+\epsilon\right]$ for each $x, y \in \mathcal{X}$. Set $\newnota{omega-epsilon}{\omega_\epsilon(x, y)}:=\sup _{\left|\theta-\theta_0\right|<\epsilon}\left|p^{\prime}(\theta, x, y)\right|$. Also assume that there exists a subset $A \subseteq \mathcal{X}$, an integer $n \geq 1, \beta>0$, 
    \chr{What is $\beta$?}\chr{What is $S$?}\jw{Notation fixed}
    and a probability measure $\varphi$ for which
    $$
    P^n(\theta, x, d y) \geq \beta \varphi(d y)
    $$
    for $x \in A, y \in \mathcal{X}$, and $\left|\theta-\theta_0\right|<\epsilon$.
\end{assumption}
\begin{assumption}\label{assumption: drift}
    For a small set $\newnota{K}{K}$, suppose there exists a measurable function $\newnota{V}{V}: X \rightarrow[1, \infty)$ and constants $0<\lambda<1, b<\infty$ such that
    \begin{equation}\label{eq: A-3-2}
    P(\theta) V(x) \leq \lambda V(x)+b \mathbbm{1}(x \in K),
    \end{equation}
    and the chains induced by $P(\theta)$ are irreducible and aperiodic for $\theta$ such that $|\theta - \theta_0| < \epsilon$. 
\end{assumption}
\begin{remark}
    If we iteratively apply the $P(\theta)$ operator on both sides of \eqref{eq: A-3-2}, assumption \ref{assumption: drift} implies that for any $n \geq 1$:
    \begin{equation}\label{eq: V-bar}
        P^n(\theta)V(x) \leq V(x)+ \frac{b}{1-\lambda} := \newnota{Bar-V}{\Bar{V}(x)}.
    \end{equation}
\end{remark}
\begin{remark}
    The Lyapunov condition assumed on function $V$ implies that
    \begin{equation}
        P(\theta)\sqrt{V(x)} \leq \sqrt{P(\theta)V(x)} \leq \sqrt{\lambda}\sqrt{V(x)} + \sqrt{b}\mathbbm{1}(x \in K).
    \end{equation}
    Since $0 < \sqrt{\lambda} < 1$ and $\sqrt{b} < \infty$, the function $\sqrt{V}$ is also a Lyapunov function.
\end{remark}
\begin{remark}
    Assumption \ref{assumption: drift} also implies that for any $g: \mathcal{X}\rightarrow \mathbb{R}$ with $|g(x)| \leq |V(x)|$ for all $x \in \mathcal{X}$, there exist constants $\newnota{R}{R} < \infty$ and $\newnota{r}{r} < 1$ such that
    \begin{equation}\label{eq: remark}
        \left|\mathbb{E}_{x}\left[g(X_n)\right] - \pi(\theta)g\right| \leq RV(x)r^n,
    \end{equation}
    and 
    \begin{equation}\label{eq: remark2}
        \left|\mathbb{E}_{x}\left[\sqrt{g(X_n)}\right] - \pi(\theta)\sqrt{g}\right| \leq R\sqrt{V(x)}r^n,
     \end{equation}
    for all $x \in \mathcal{X}$ and $n \geq 1$.
\end{remark}
\begin{assumption}\label{assumption-starting}
    We also assume our starting distribution $\mu_0$ satisfies that $\mu_0 V < \infty$.
\end{assumption}
\begin{assumption}\label{assumption-coupling}
    Let $\newnota{tau-xz}{\tau_{x, z}}$ be the coupling time of two chains starting from state $x$ and state $z$ respectively. Assume the coupling time satisfies a geometrically decaying tail probability, i.e.
    \begin{equation}\label{eq: A-tail}
        P(\tau_{x, z} > t) \leq M\left(V(x) + V(z)\right)\rho^t,
    \end{equation}
    for some $R < \newnota{M}{M} < \infty$ and $r < \newnota{rho}{\rho} <1$.
\end{assumption}
Assumption \ref{assumption: r-g-A4} is carried directly from \cite{rhee2017lyapunov}. 
Assumption \ref{assumption: drift} sometimes can be used to verify Assumption \ref{assumption-coupling}. 
See Section 3.2 in \cite{jacob2020unbiased} for additional assumptions and details to do this. 
Assumption \ref{assumption-coupling} can be extended to polynomially decaying tail probabilities, see \cite{douc2022solving}\cite{atchade2024unbiased}.

\subsection{Unbiased Estimation for the Stationary Mean}\label{sec: stationary-mean}
The intuitive arguments that lead to an unbiased estimator for the stationary mean starts from a telescoping sum technique \cite{glynn2014exact}. Under some regularity condition that guarantee exchange of limits, for any $\newnota{k}{k}\geq 0$:
\chr{under some regularity condition that guarantee exchange of limits}\jw{added}
\begin{align}
    \gamma(\theta) &= \lim_{t \rightarrow \infty}\mathbb{E}\left[f(X_t)\right]\label{eq: begin-g}\\
    &= \mathbb{E}\left[f(X_k)\right] + \sum_{t=1}^\infty \left(\mathbb{E}\left[f(X_{k + tL})\right] - \mathbb{E}\left[f(X_{k+(t-1)L})\right]\right)\\
    &= \mathbb{E}\left[f(X_k)\right] + \sum_{t=1}^{\infty} \left(\mathbb{E}\left[f(X_{k + tL})\right] - \mathbb{E}\left[f(Y_{k+(t-1)L})\right]\right)\\
    &= \mathbb{E}\left[f(X_k)  + \sum_{t=1}^{\floor*{\frac{\tau^L-k}{L}}} \left(f(X_{k + tL}) - f(Y_{k+(t-1)L})\right)\right],\label{eq: end-g}
\end{align}
\chr{above equation: parentheses in the summation}\jw{parentheses added}
where the third equality follows from the fact that both $X$ and $Y$ evolves according $P(\theta)$ and the last equality follows from that $X_{i+L}=Y_{i}$ for all $i \geq \tau^L$. 
\chr{$X_{i+1}$?}\jw{index fixed}
Therefore,
\begin{equation}\label{eq: H-kL}
    \newnota{Hk}{H_k^L} := f(X_k)  + \sum_{t=1}^{\floor*{\frac{\tau^L-k}{L}}} \left(f(X_{k + tL}) - f(Y_{k+(t-1)L})\right)
\end{equation}
would be an unbiased estimator for $\gamma(\theta)$ for any $k \geq 0$ if \eqref{eq: begin-g} to \eqref{eq: end-g} hold rigorously and so is an average of $H_k^L$ for different $k$ values $\newnota{H-km}{H_{k:m}^L} :=\frac{1}{m-k+1} \sum_{t=k}^m H_t^L$, where $\newnota{m}{m} \geq k$. We may rewrite it as
\begin{equation}\label{eq: H-kmL}
H_{k:m}^L =\frac{1}{m-k+1} \sum_{t=k}^m f(X_t) + \frac{1}{m-k+1} \sum_{t=k+L}^{\tau^L-1} c_t \left(f(X_t) - f(Y_{t-L})\right),
\end{equation}
where 
\begin{equation}
    \newnota{ct}{c_t} = \floor*{\frac{\min(m+L, t) - k - (t - k)\% L}{L}}.
\end{equation}
If $L = 1$, then $H_{k:m}^L$ is equivalent to the stationary mean estimator developed in \cite{jacob2020unbiased}. However, choosing a larger $L$ is proposed and advocated in \cite{vanetti2020} as it generally improves the performance of the estimator, sometimes dramatically.
The intuition behind choosing a larger $L$ is that the coefficients $c_t$ in \eqref{eq: H-kmL} can be greatly reduced and make $H_{k:m}^L$ behave more like a long-run average.
See \cite{douc2022solving}, \cite{atchade2024unbiased} and \cite{ruiz2021unbiased} for more detailed analysis and examples of $H_{k:m}^L$.

\subsection{Unbiased Estimation for the solution to the Poisson's Equation}
Let $\newnota{g}{g}$ denote the solution to the Poisson's equation such that for any $x \in \mathcal{X}$, the following equality holds
\begin{equation}\label{eq: Poisson}
    g(x) - P(\theta)g(x) = f(x) - \pi(\theta)f,
\end{equation}
The solution $g$ has been shown to play an important role in analyzing the sum of the form $\sum_{t=0}^{n}f(X_t)$ for Markov chains \cite{glynn2024solution}.
Past works have analyzed the uniqueness of the solution and its connections to the Lyapunov conditions, see \cite{bhulai2003uniqueness}, \cite{glynn1996liapounov} and \cite{rhee2017lyapunov}. 
The most basic form the solution function $g$ could take is termed the \textbf{fundamental solution} to the Poisson's equation and is defined as:
\begin{equation}
    \newnota{g-funda}{g^{fu}(x)} : = \mathbb{E}_x^{\theta}\left[\sum_{t=0}^{\infty}\left(f(X_t) - \pi(\theta)f\right)\right].
\end{equation}
Note that \eqref{eq: Poisson} can admit various solutions that differ from each other by constants. 
One of the most exploited expression of the solution $g$ depends on the regeneration structure of the analyzed Markov chain and attempt to estimate the following form:
\begin{equation}
    \newnota{g-regen}{g^{re}_{\alpha}(x)} : = \mathbb{E}_x^{\theta}\left[\sum_{t=0}^{T(\alpha)}\left(f(X_t) - \pi(\theta)f\right)\right]
\end{equation}
where $\newnota{alp}{\alpha}$ is a regeneration state and $\newnota{T-alp}{T(\alpha)}:= \inf\{t > 0: X_t = \alpha\}$ is the regeneration time, e.g. in \cite{dai2022queueing}.
When the state space is uncountable, sometimes we are still able to construct regeneration times using a splitting technique \cite{nummelin1978splitting}.
However, the constructed regeneration time could be unnecessarily long causing a large variance for the estimator.
Since we do not assume the existence or the ability to construct efficient regeneration times, we focus on a different form that takes advantage of the existence of a coupling kernel, first introduced in \cite{douc2022solving}.
For a pre-specified state $z \in \mathcal{X}$, define:
\begin{equation}
\begin{aligned}
    \newnota{g-z}{g_z(x)} &:= g^{fu}(x) - g^{fu}(z)\\
    &=  \mathbb{E}_x^{\theta}\left[\sum_{t=0}^{\infty}\left(f(X_t) - \pi(\theta)f\right)\right] -  \mathbb{E}_z^{\theta}\left[\sum_{t=0}^{\infty}\left(f(X_t) - \pi(\theta)f\right)\right]\\
    &=  \mathbb{E}_{x, z}^{\theta}\left[\sum_{t=0}^{\infty}\left(f(X_t) - f(Z_t)\right)\right]\\
    &= \mathbb{E}_{x, z}^{\theta}\left[\sum_{t=0}^{\tau_{x,z}}\left(f(X_t) - f(Z_t)\right)\right]
\end{aligned}
\end{equation}
where $(X_t, Z_t)_{t=0,1,2,\dots}$ is a joint Markov chain that starts from $(x, z)$ and evolves according to $\Bar{P}(\theta)$ and $\tau_{x,z} = \inf\{t \geq 0: X_t = Z_t\}$.
Thus naturally $\sum_{t=0}^{\tau_{x,z}}\left(f(X_t) - f(Z_t)\right)$ would be an unbiased estimator for $g_z(x)$ which is a solution to \eqref{eq: Poisson}.

\section{Unbiased Derivative Estimators \jw{Ready to be reviewed}}\label{sec-UDE}
\subsection{Derivative estimator with stationary mean approach}\label{subsec-sma}
Let $\newnota{theta-0}{\theta_0}$ denote our current parameter at which we are trying to compute the derivative.
In this section, we introduce conditions ensuring differentiability of the steady-state mean and establish a probabilistic representation for its derivative $\gamma'(\theta_0)$, upon which we construct unbiased derivative estimators.
\begin{proposition}\label{proposition-differentiability}
    (Theorem 4.1 in \cite{rhee2017lyapunov}) With Assumption \ref{assumption: r-g-A4}, let $\kappa: \mathbb{R}_{+} \rightarrow \mathbb{R}_{+}$be a function for which $\kappa(x) \geq x$ and $\kappa(x) / x \rightarrow \infty$ as $x \rightarrow \infty$. Suppose that there exist positive constants $\epsilon, c_0$, and $c_1$, and non-negative finite-valued functions $q, v_0$, and $v_1$ for which
    $$
    \begin{aligned}
    & \left(P(\theta) v_0\right)(x) \leq v_0(x)-(q(x) \vee 1)+c_0 \mathbb{I}(x \in A) \\
    & \left(P(\theta) v_1\right)(x) \leq v_1(x)-\kappa\left(\int_S\left(1 \vee \omega_\epsilon(x, y)\right)\left(v_0(y)+1\right) P(\theta, x, d y)\right)+c_1 \mathbb{I}(x \in A)
    \end{aligned}
    $$
    for $x \in \mathcal{X},\left|\theta-\theta_0\right|<\epsilon$, and
    $$
    \sup _{x \in A} v_0(x)<\infty.
    $$
    Then
    \begin{equation}\label{eq: gp0}
        \gamma'(\theta_0) = \mathbb{E}^{\theta_0}_{\pi(\theta_0)}\left[p'(\theta_0, X_0, X_1)g(X_1)\right].
    \end{equation}
    where $g$ is a solution to \eqref{eq: Poisson}. 
\end{proposition}
Refer to Section 4 in \cite{rhee2017lyapunov} for more details and intuitions.
Essentially, equation \eqref{eq: gp0} expresses the gradient in \eqref{eq: gamma-prime} as the stationary expectation of a suitably defined function with respect to the invariant measure $\pi(\theta_0)$.
It is noteworthy that Assumption \ref{assumption: drift}-\ref{assumption-coupling} are not explicitly required in Proposition \ref{proposition-differentiability}; however, the geometric ergodicity implied by these assumptions becomes instrumental for subsequent efficiency analyses.
Here we focus on estimating the derivative based on the probabilistic representation \eqref{eq: gp0}.
First we consider $g_z$ in place of $g$ in \eqref{eq: gp0} and define 
\begin{equation}\label{eq: G_z}
    \newnota{G-z}{G_z(x)} := \sum_{t=0}^{\tau_{x,z}}\left(f(X_t) - f(Z_t)\right)
\end{equation}
as an unbiased estimator for $g_z(x)$.
Observe that by conditioning on $X_0$, we may rewrite \eqref{eq: Poisson} to
\begin{equation}
    \gamma'(\theta_0) = \mathbb{E}^{\theta_0}_{\pi(\theta_0)}\left[p'(\theta_0, X_0, X_1)g(X_1)\right] = \mathbb{E}^{\theta_0}_{\pi(\theta_0)}\left[h(X_0)\right],
\end{equation}
where
\begin{equation}
    \newnota{h-x}{h(x)} := \mathbb{E}_{x}^{\theta_0}\left[p'(\theta_0, x, X_1)g(X_1)\right].
\end{equation}
Now we define
\begin{equation}
    \newnota{H-x0x1}{H(x_0, x_1)} = p'(\theta_0, x_0, x_1)G_z(x_1),
\end{equation}
where $G_z(x_1)$ is independent given $x_1$ and we make the dependence on choice of state $z$ implicit in $H(x_0, x_1)$ and whenever obvious later.
Building upon the unbiased stationary mean estimator introduced in \eqref{eq: H-kL}, we propose our initial derivative estimator, defined explicitly as follows:
\begin{equation}\label{eq: H-kL-1}
    \newnota{Hk-1}{H_{1}^{k,L}(X, Y)} := H(X_k, X_{k+1})  + \sum_{t=1}^{\floor*{\frac{\tau^L-k}{L}}} \left(H(X_{k + tL}, X_{k + tL + 1}) - H(Y_{k+(t-1)L}, Y_{k+(t-1)L + 1})\right).
\end{equation}
Under suitable technical conditions (to be specified in subsequent sections), $H_{1}^{k,L}(X, Y)$ can be rigorously shown to be an unbiased estimator for $\gamma'(\theta_0)$ with finite second moment and finite computation time.
The formal proof will be delayed to later this section where we prove the result for a more generalized version of the derivative estimator.
A fundamental difference between $H_{1}^{k,L}(X, Y)$ and previously studied stationary mean estimators is the non-trivial computational complexity involved in evaluating $H(x_0, x_1)$, which necessitates simulating a coupled Markov chain until coalescence.
Furthermore, for fixed inputs $(x_0, x_1)$, the quantity $H(x_0, x_1)$ itself is inherently random, thereby potentially introducing substantial variance, particularly when coupling times are prolonged due to slow mixing.
Therefore, to make a more practical estimator, we need machinery to improve efficiency for each of the $H(x_0, x_1)$ terms in \eqref{eq: H-kL-1}.

\subsection{Derivative estimator with the $L$-skeleton chain approach}
We note that $\gamma(\theta_0)$ stays the same when we consider the $L$-skeleton chain and so should its derivative $\gamma'(\theta_0)$.
We first present an intuitive derivation of the derivative representation for the $L$-skeleton chain,  subsequently formalizing it in a rigorously stated theorem.
First see that for the $L$-step transition kernel $\newnota{P-L}{P^L(\theta)}:= \left(P(\theta, x, dy): x, y \in \mathcal{X}\right)$,
\begin{align}
    P^L(\theta, x_0, dx_L) &= \int_{\mathcal{X}}P^{L-1}(\theta, x_0, dx_{L-1})P(\theta, x_{L-1}, dx_{L})dx_{L-1}\\
    &= \int_{\mathcal{X}} \left(\int_{\mathcal{X}} P^{L-2}(\theta, x_0, dx_{L-2})P(\theta, x_{L-2}, dx_{L-1})dx_{L-2}\right) P(\theta, x_{L-1}, dx_{L})dx_{L-1}\\
    &= \int_{\mathcal{X}^{L-1}} \prod_{i=0}^{L-1} P(\theta, x_i, dx_{i+1}) dx_1 dx_2 \cdots dx_{L-1}. 
\end{align}
Then, assume the validity for interchanging derivatives and integrals, we may rewrite \eqref{eq: gp0} as 
\begin{align}
    \gamma'(\theta_0) &= \mathbb{E}^{\theta_0}_{\pi(\theta_0)}\left[\frac{\frac{d}{d\theta}P^{L}(\theta, X_0, dX_L)|_{\theta_0}}{P^L(\theta_0, X_0, dX_L)}g^L(X_{L}) \right]\\
     &= \mathbb{E}_{X_0\sim\pi(\theta_0)}\left[\mathbb{E}\left[\frac{\frac{d}{d\theta}P^{L}(\theta, X_0, dX_L)|_{\theta_0}}{P^L(\theta_0, X_0, dX_L)}g^L(X_{L})|X_0\right] \right]\\
    &= \mathbb{E}_{X_0\sim\pi(\theta_0)}\left[\int \frac{d}{d\theta}P^{L}(\theta, X_0, dX_L)|_{\theta_0}g^L(X_L)\right]\\
    &=\mathbb{E}^{\theta_0}_{\pi(\theta_0)}\left[\left(\sum_{i=0}^{L-1}p'(\theta_0, X_{i}, X_{i+1})\right)g^L(X_L)\right],
\end{align}
where the first equality is \eqref{eq: gp0} with the $L$-step kernel $P^L(\theta)$ and for now is assumed to hold without further assumptions; the last equality comes from applying chain rule and divide and multiply $\prod_{i=0}^{L-1} P(\theta_0, x_i, dx_{i+1})$, and $\newnota{g-L}{g^L}$ is a solution to the Poisson's equation associated with the $L$-skeleton Markov chain that satisfies the following equation
\begin{equation}\label{eq: Poisson-2}
    g^{L} - P^{L}(\theta)g^{L} = f - \pi(\theta)f.
\end{equation}
We next rigorously demonstrate that the $L$-skeleton representation of the derivative holds without imposing additional assumptions beyond those stated in Proposition \ref{proposition-differentiability}.
This approach circumvents the typically intricate verification of conditions directly on the $L$-step transition kernel.
\begin{theorem}\label{theorem: L-representation}
    Assume that \eqref{eq: gp0} holds, then
    \begin{equation}\label{eq: gp1}
        \gamma'(\theta_0) = \mathbb{E}^{\theta_0}_{\pi(\theta_0)}\left[\left(\sum_{i=0}^{L-1}p'(\theta_0, X_{i}, X_{i+1})\right)g^L(X_L)\right],
    \end{equation}
    for all $L \geq 1$.
\end{theorem}
The proof is deferred to Section \ref{sec-proof-thm-L-representation}.
Naturally, here for a pre-specified state $z \in \mathcal{X}$, we develop the $L$-skeleton chain counterpart for \eqref{eq: G_z} as:
\begin{equation}
    \newnota{G-z-L}{G^{L}_z(x)}:=\sum_{i=0}^{\floor*{\frac{\tau_{x, z}}{L}}}\left(f(X_{iL}) - f(Z_{iL})\right),
\end{equation}
and for the same reason as before, this is an unbiased estimator for
\begin{equation}
    \newnota{g-z-L}{g^{L}_z(x)} := \mathbb{E}_{x, z}^{\theta_0}\left[\sum_{i=0}^\infty \left(f(X_{iL}) - f(Z_{iL})\right)\right],
\end{equation}
where $g^{L}_z$ is a solution to \eqref{eq: Poisson-2}. 
Analogously to the development in Section \ref{subsec-sma}, we first define the following intermediate quantities:
\begin{equation}\label{eq: H2}
    \newnota{H-0-L}{H(\{x_i\}_{i=0}^{L})}: = \left(\sum_{i=0}^{L-1}p'(\theta_0, x_{i}, x_{i+1})\right)G_z^L(x_L),
\end{equation}
and then we can define our second derivative estimator to be
\begin{equation}\label{eq: H-kL-2}
    \newnota{Hk-2}{H_{2}^{k,L}(X, Y)} := SE_{k}^{L} + BC_{k}^{L}
\end{equation}
where 
\begin{align}
    \newnota{SE-kL}{SE_{k}^{L}} &:= H\left(\{X_i\}_{i=k}^{k+L}\right);\\
    \newnota{BC-kL}{BC_{k}^{L}} &:= \sum_{t=1}^{\floor*{\frac{\tau^L-k}{L}}} H\left(\{X_i\}_{i=k+tL}^{k+(t+1)L}\right) - H\left(\{Y_i\}_{i=k+(t-1)L}^{k+tL}\right).
\end{align}
represents the singleton estimator (SE) and and bias correction (BC) term to the singleton estimator respectively.
The motivation to develop this estimator from the $L$-skeleton chain is the greatly reduced variance when considering $G^{L}_z$ instead of $G_z$ since $G^{L}_z$ is a sum of much fewer random variables when $L$ is large.
The price we pay for the reduced variance in $G^{L}_z$ is the increased variance in the $\sum_{i=0}^{L-1}p'(\theta_0, x_{i}, x_{i+1})$ term in \eqref{eq: H2} as $L$ gets larger.
However, under appropriate conditions, we can show that $\{M_n\}_{n=0,1,2}$ with $M_0 = 0$ and 
\begin{equation}\label{eq: Mn}
    \newnota{M-L}{M_n} := \sum_{i=0}^{L-1}p'(\theta_0, x_{i}, x_{i+1})
\end{equation}
for $n \geq 1$ is a martingale.
Thus the variance increase for the martingale term usually gets outweighed by the sometimes dramatic variance decrease in the estimation of the solution to the Poisson's equation associated with the $L$-skeleton chain and causing a great improvement in efficiency.
Nevertheless, it is unwise to choose $L$ to be arbitrarily large as the estimator's efficiency does not monotonically improve as $L$ gets larger.
The analysis of the efficiency and a guideline for choosing $L$ will be discussed in the next section.
Here we first prove our result rigorously.

\begin{assumption}\label{assumption-moment}
    Suppose there exists a $\newnota{p}{p} > 1$ , $\newnota{kappa}{\kappa} > 1$ and a $\newnota{delta}{\delta} > 0$ such that $\left|f^{2\kappa p+\delta}\right|_{V},  \left|\Omega^{\frac{2\kappa p}{p-1}}\right|_{V} < \infty$, where
    \begin{align}
        \newnota{Omega}{\Omega^{\frac{2\zeta p}{p-1}}(x)} &= \int_{\mathcal{X}}\omega_{\epsilon}(x, y)^{\frac{2\zeta p}{p-1}}P(\theta_0, x, dy),\\
        \newnota{f2p}{f^{2\zeta p+\delta}(x)} &= f(x)^{2\zeta p+\delta}.
    \end{align}
\end{assumption}

\begin{lemma}\label{lemma: Gamma}
    With assumptions \ref{assumption: r-g-A4}-\ref{assumption-moment}, if we define for $1 \leq \zeta \leq \kappa$:
    \begin{equation}
        \newnota{Gamma-x}{\Gamma^{L}_{\zeta}(x)} := \mathbb{E}_x^{\theta_0}\left[H\left(\{X_i\}_{i=0}^{L}\right)^{2\zeta}\right],
    \end{equation}
    then we have $\Gamma^{L}_{\zeta}(x) \leq V(x)U_{\zeta, z}(L)$ where
    \begin{equation}
        \newnota{U-zeta}{U_{\zeta, z}(L)} := \left(L^{\zeta}A_{z} + L^{\zeta}\left(\frac{\left(\rho^{\frac{\delta}{2\zeta p(2\zeta p+\delta)}}\right)^L}{1-\left(\rho^{\frac{\delta}{2\zeta p(2\zeta p+\delta)}}\right)^L}\right)^{2\zeta}B_{z}\right),
    \end{equation}
    and
    \begin{align}
        \newnota{A_z}{A_z^{\zeta}} &:= 2^{4\zeta - 2}\left|f^{2\zeta p+\delta}\right|_{V}^{\frac{2\zeta}{2\zeta p+\delta}}\left(C_{\frac{2\zeta p}{p-1}}|\Omega^{\frac{2\zeta p}{p-1}}|_{V} \left(1 + \frac{b}{1-\lambda}\right)\right)^{\frac{p-1}{ p}}\left(1 +\left(\frac{b}{1-\lambda}\right)^{\frac{2\zeta}{2\zeta p+\delta}}+\Bar{V}(z)^{\frac{2\zeta}{2\zeta p+\delta}}\right),\\
        \newnota{B_z}{B_z^{\zeta}} &:= A_z^{\zeta} \cdot M^{\frac{2\zeta\delta}{2\zeta p(2\zeta p+\delta)}}\left(1 + \left(\frac{b}{1-\lambda}\right)^{\frac{2\zeta\delta}{2\zeta p(2\zeta p+\delta)}}+V(z)^{\frac{2\zeta\delta}{2\zeta p(2\zeta p+\delta)}}\right).
    \end{align}
    where $\newnota{C-kappa}{C_{l}} := \left[8(l-1) \max(1, 2^{l-3})\right]^{l}$. One direct implication is that $\left|\Gamma^{L}_{\zeta}\right|_{V} \leq U_{\zeta, z}(L) < \infty$.
\end{lemma}
The proof is in Section \ref{sec-proof1}.
Next we prove the formal results for our estimator $H_{2}^{k,L}(X, Y)$.
\begin{theorem}\label{theorem: H2}
    Under assumptions \ref{assumption: r-g-A4}, \ref{assumption: drift}, \ref{assumption-starting}, \ref{assumption-coupling} and \ref{assumption-moment}, $H_{2}^{k,L}(X, Y)$ is an unbiased estimator for $\gamma'(\theta)$ with finite second moment and finite expected computation time. Furthermore,
    \begin{equation}\label{eq: thm-bc}
        \mathbb{E}_{\mu_0}^{\theta_0}\left[\left(BC_{k}^L\right)^2\right] \rightarrow 0
    \end{equation}
    and
    \begin{equation}\label{eq: thm-se}
        \left|\mathbb{E}_{\mu_0}^{\theta_0}\left[\left(SE_{k}^L\right)^2\right] - \pi(\theta_0)\Gamma^{L}_{1}\right| \rightarrow 0
    \end{equation}
    as $k \rightarrow \infty$.
\end{theorem}
The proof of the Theorem is in Section \ref{sec-proof2}.
Theorem \ref{theorem: H2} makes rigorous some obvious intuitions.
The second moment (and thus the first moment) goes to 0 as the starting point $k$ increases since $X_k$ becomes closer and closer to being distributed as $\pi(\theta_0)$.
Thus we may choose $k$ to be large such that the $SE_k^L$ term dominates and the variance of the estimator converges to the variance of the singleton estimator if we were to start the chain from stationary $\pi(\theta_0)$.
On the other hand, even though choosing a large burn-in period $k$ may reduce variance, it is not economical to just generate a one sample estimator and discard all previously simulated chains.
Therefore, a running average version of the estimator may be of greater interest.

\subsection{Averaging the $L$-skeleton chain estimator}
Several prior studies on stationary expectation estimation advocate for running-average versions of their estimators (e.g., \cite{jacob2020unbiased}, \cite{douc2022solving}, \cite{atchade2024unbiased}).
Nevertheless, in our current context, the computational overhead associated with evaluating the functional of interest is substantial, warranting a modified approach.
To be specific, we need to independently run two chains until they couple every time we are computing $G_z$.
Assume we define a running average version similar to \eqref{eq: H-kmL}, we will get for some $m \geq k$,
\begin{equation}
    \newnota{H-3}{H_3^{k, m, L}(X, Y)} := \frac{1}{m-k+1}\sum_{t=k}^{m}H_2^{t, L}(X, Y).
\end{equation}
A notable issue arising with this running-average approach is the potentially high correlation between successive estimators, namely $H_2^{t,L}(X,Y)$ and $H_2^{t+1,L}(X,Y)$.
Assume that we take $k$ to be a large quantile of the coupling time $\tau^L - L$ so that the bias correction term plays little role \cite{jacob2020unbiased}, then notice that
\begin{align}
    H_2^{t, L}(X, Y) &\approx H\left(\{X_i\}_{i=t}^{t+L}\right) = \left(\sum_{i=t}^{t + L - 1}p'(\theta_0, X_{i}, X_{i+1})\right)G_z^L(X_{t+L});\\
    H_2^{t+1, L}(X, Y) &\approx H\left(\{X_i\}_{i=t+1}^{t + 1 + L}\right) = \left(\sum_{i=t+1}^{t + L}p'(\theta_0, X_{i}, X_{i+1})\right)G_z^L(X_{t+L+1}),
\end{align}
where those two consecutive estimators can be strongly correlated even with moderate choice of $L$ since they share a great number of terms in the summation, specifically $\sum_{i=t + 1}^{t + L - 1}p'(\theta_0, X_{i}, X_{i+1})$ appears in both of the consecutive estimators.
Also, $G_z^L(X_{t+L})$ and $G_z^L(X_{t+L+1})$ can also be correlated since $X_{t+L}$ and $X_{t+L+1}$ are one step away.
Meanwhile, two independent coupled chains needs to be simulated to compute $G_z^L(X_{t+L})$ and $G_z^L(X_{t+L+1})$, causing a significant computational burden.
It is reasonable to question the merit of expending substantial computational effort to generate highly correlated estimates.
Owen (2017) \cite{owen2017statistically} addresses precisely this issue, demonstrating that when function evaluation costs dominate the computational cost of advancing the Markov chain, strategically thinning samples can notably enhance estimation efficiency.
Thus here we propose a more general form of the running average estimator that utilizes the thinning technique in \cite{owen2017statistically}.
For some $m \geq 1$, we define:
\begin{equation}
    \newnota{H-4}{H_4^{k, m, L}(X, Y)} := \frac{1}{m}\sum_{t=0}^{m-1}H_2^{k+tL, L}(X, Y).
\end{equation}
Note that for each of the $t = 1, 2,\dots, m-1$, Theorem \ref{theorem: H2} applies to $H_2^{k+tL, L}(X, Y)$.
Since $H_4^{k, m, L}(X, Y)$ is an average of them, a simple application of Minkowski's inequality would yield that $H_4^{k, m, L}(X, Y)$ also has finite second moment and thus it is unbiased and has computation time.
Similarly, if $k$ is chosen to be a large quantile of $\tau^L - L$, then the average of the singleton estimators is going to dominate and
\begin{equation}\label{eq: H-4-approx}
    H_4^{k, m, L}(X, Y) \approx \frac{1}{m}\sum_{t=0}^{m-1} H\left(\{X_i\}_{i=k+tL}^{k+(t+1)L}\right),
\end{equation}
where we usually take $m$ to be large so that the burn-in cost $k$ becomes less significant.
The asymptotic variance of $H_4^{k, m, L}(X, Y)$ is then approximately $\sigma^2_{L} + 2\sum_{j=1}^{\infty}\sigma_{j, L}$ where 
\begin{equation}
    \begin{aligned}
        \newnota{sig-2}{\sigma^2_{L}} &:= Var_{\pi(\theta_0)}\left(H\left(\{X_i\}_{i=0}^{L}\right)\right),\\
        \newnota{rho-jL}{\sigma_{j, L}} &:= Cov_{\pi(\theta_0)}\left(H\left(\{X_i\}_{i=0}^{L}\right), H\left(\{X_i\}_{i=jL}^{(j+1)L}\right)\right).
    \end{aligned}
\end{equation}
That is,
\begin{equation}\label{eq: H-4-approx-2}
    m \cdot Var(H_4^{k, m, L}(X, Y)) \approx \sigma^2_{L} + 2\sum_{j=1}^{\infty}\sigma_{j, L}.
\end{equation}
Observe that consecutive estimators such as $H({X_i}{i=0}^{L})$ and $H({X_i}{i=L}^{2L})$ do not share overlapping terms from the chain $X$, implying intuitively low correlation provided $L$ is chosen sufficiently large.
A more rigorous quantitative analysis of estimator efficiency will be presented in the subsequent section.

\section{Efficiency Analysis and Parameter Selections \jw{Ready to be reviewed}}\label{sec-EA}
In this section, we conduct a detailed efficiency analysis of our estimator, with inefficiency quantified by the work-variance product.
Leveraging the outcomes of this analysis, we provide explicit guidelines for selecting key parameters, including the lagging parameter LL and the reference state $z$.
Concretely, if we choose $k$ and $m$ as chosen according to recommendations in the previous section such that approximations \eqref{eq: H-4-approx} and \eqref{eq: H-4-approx-2} hold, then the expected number of Markov transitions to produce one $H_4^{k, m, L}(X, Y)$ is approximately $k + m(L + 2\mathbb{E}[\tau_{\pi(\theta_0), z}])$ where $\newnota{tau-pi}{\tau_{\pi(\theta_0), z}}$ is the coupling time of two chains starting from stationary distribution $\pi(\theta_0)$ and a pre-specified state $z$.
For specific breakdowns, $k$ is the burn-in time for chain $X$; $L + 2\mathbb{E}[\tau_{\pi(\theta_0), z}]$ is the cost to advance from $X_{k+tL}$ to $X_{k+(t+1)L}$ plus the expected computational cost to generate independently a pair of coupled chain to estimate the solution to the Poisson's equation $G_z^L$.
The multiplier $m$ represents those steps will be repeated for $m$ times.
Thus, our analysis focuses on the following work-variance product:
\begin{align}
    \left(k + m\left(L + 2\mathbb{E}[\tau_{\pi(\theta_0), z}]\right)\right)Var(H_4^{k, m, L}(X, Y)) &\approx \left(\frac{k}{m} + L + 2\mathbb{E}[\tau_{\pi(\theta_0), z}]\right)\left(\sigma^2_{L} + 2\sum_{j=1}^{\infty}\sigma_{j, L}\right)\\
    &\approx \left(L + 2\mathbb{E}[\tau_{\pi(\theta_0), z}]\right)\left(\sigma^2_{L} + 2\sum_{j=1}^{\infty}\sigma_{j, L}\right)
\end{align}
where the approximation follows from \eqref{eq: H-4-approx-2}. 
\begin{lemma}\label{lemma: efficiency}
    With Assumptions \ref{assumption: r-g-A4}-\ref{assumption-moment}, the asymptotic variance is bounded by
    \begin{equation}
        \left(\sigma^2_{L} + 2\sum_{j=1}^{\infty}\sigma_{j, L}\right) \leq U_{1, z}(L)\pi(\theta_0)V\left(3 + \frac{2M\rho^L}{1-\rho^L}\right).
    \end{equation}
\end{lemma}
The proof is recorded in Section \ref{sec: proof-lemma3}.
Given this quantified upper bound on the asymptotic variance, our objective now shifts to establishing general parameter-selection guidelines aimed at minimizing the work-variance product of our estimator.\\
\textbf{Choice of $z$:} First we see that
\begin{equation}\label{eq: Etau}
    \begin{aligned}
        P_{X_0\sim\pi(\theta_0)}\left(\tau_{X_0, z} > t\right) &= \int_{\mathcal{X}}P\left(\tau_{y, z} > t\right)\pi(\theta_0, dy)\\
        &\leq  \int_{\mathcal{X}}M(V(y) + V(z))\rho^{t}\pi(\theta_0, dy)\\
        &\leq M\left(\pi(\theta_0)V+V(z)\right)\rho^t.
    \end{aligned}
\end{equation}
Therefore $\mathbb{E}\left[\tau_{\pi(\theta), z}\right] \leq \frac{M}{1-\rho}\left(\pi(\theta_0)V+V(z)\right)$.
Furthermore, note that the upper bound $U_{1, z}(L)$ also decreases monotonically as $V(z)$ decreases.
Thus the most obvious choice of $z$ is to make it a state such that $V(z)$ is as small as possible.
It makes sense in most cases as a smaller $V(z)$ usually entails that $z$ is in a more frequently visited region, reflecting more typical behavior of the Markov chain. \\
\textbf{Choice of $L$:} In contrast, the selection of the lagging parameter $L$ is a little more intricate, as it inherently depends on less readily available quantities such as the smallest geometric convergence rate $\rho$ satisfying our theoretical assumptions.
Nevertheless, we propose practical guidelines relying solely on accessible information, demonstrating that employing the $L$-skeleton approach can yield performance improvements by at least a certain factor compared to the naive case $L = 1$.
\begin{theorem}\label{theorem: efficiency}
    Suppose we define
    \begin{equation}
        \newnota{W-L}{W(L)} := \left(L + 2\mathbb{E}[\tau_{\pi(\theta_0), z}]\right)U_{1, z}(L)\pi(\theta_0)V\left(3 + \frac{2M\rho^L}{1-\rho^L}\right),
    \end{equation}
    then
    \begin{equation}
        \frac{W(1)}{W(\mathbb{E}[\tau_{\pi(\theta_0), z}] - 1)} > \frac{\mathbb{E}[\tau_{\pi(\theta_0), z}] - 1}{3}.
    \end{equation}
\end{theorem}
The proof is recorded in Section \ref{sec: proof-theorem4}. 
Theorem \ref{theorem: efficiency} suggests that if we select $L = \mathbb{E}[\tau_{\pi(\theta_0), z}] - 1$, then the work-variance product upper bound would be reduced to at most $\left(\mathbb{E}[\tau_{\pi(\theta_0), z}] - 1\right)/3$ fraction of the original one with $L = 1$. 
In scenarios involving slow-mixing Markov chains, where coupling times tend to be substantially prolonged, such a guideline provides a significant efficiency improvement. 
This guideline is practical, as it relies exclusively on estimating the expected coupling time, a quantity that is typically straightforward to approximate empirically. 
The theorem suggests that the improvement is at least this much and we encourage using it as a starting point and trying larger $L$ values.
While the choice $L = \mathbb{E}[\tau_{\pi(\theta_0), z}] - 1$ is rarely optimal for minimizing inefficiency, it serves effectively as an initial reference point.
Practitioners are encouraged to explore larger values of $L$ to potentially achieve further performance gains.\\
For the previously developed unbiased estimators for the regular stationary mean as described in Section \ref{sec: stationary-mean}, the choice of $L$ is suggested to be a large quantile of the coupling time which is the same as the guideline for selecting $k$ and such choice would incur very minor cost \cite{atchade2024unbiased}.
In our setting, the choice should be made with more care since the $L^2A_z^{1}$ term in $W(L)$ increases with square rate. 
Nonetheless, for slow-mixing Markov chains (i.e., when $\rho$ is close to 1), it generally remains advantageous to err on the side of selecting larger values of $L$.
In such cases, the exponential decay of the term $\frac{\rho^L}{1-\rho^L}$ greatly outweighs the linear increase of $L$, ensuring improved overall estimator efficiency.

\section{Transition kernel represented by neural networks \jw{Ready to be reviewed}}\label{sec-application}
One significant advantage of our proposed estimator is its ease of implementation and broad applicability to neural network (NN)-based parametrizations.
Although likelihood-ratio-based methods can also be applied to NN parametrizations, they often require restrictive conditions, such as identifiable regenerative structures or exact knowledge of the stationary mean—assumptions rarely met in practice.
In this section, we present two fundamental scenarios demonstrating the applicability of our gradient estimator without imposing any additional structural requirements.

\subsection{Policies parametrized by Neural Networks}
Consider the reinforcement learning (RL) setting where we can control the Markov Decision Process (MDP) by changing the parametrized policy. 
Given current state $s$, we abuse the notation a little in this section and denote for now that $\pi_{\theta}(a|s)$ as a policy represented by a neural network where $\theta$ is the weights and biases of the network.
$\pi_\theta(\cdot|s)$ outputs a probability distribution on the action space $\mathcal{A}$.
The system dynamics, represented by $P(\cdot|s, a)$ will give the probability distribution for the next state $s'$.
The system dynamics here is assumed to be known or learned. 
In this case,
\begin{equation}
    P(\theta, s, s') = \int_{\mathcal{A}} \pi_{\theta}(a|s)P(s'|s, a),
\end{equation}
and if we may interchange the derivative and the integration,
\begin{equation}
    \frac{d}{d\theta}P(\theta, s, s') = \int_{\mathcal{A}} \frac{d}{d\theta} \pi_{\theta}(a|s)P(s'|s, a).
\end{equation}
If the action space $\mathcal{A}$ is finite (such as in a preemptive queuing control problem), then given $(s, s')$, both of the above expressions can be evaluated exactly as $\pi_{\theta}(a|s)$ can be obtained by forward feeding the neural network and $\frac{d}{d\theta} \pi_{\theta}(a|s)$ can be obtained by back-propagation. 
Thus the oracle to numerically compute transition densities and their derivatives is available.
When the action space is uncountable or infinite, we may consider making the state-action pair $(s, a)$ as a state of a markov chain on the joint state-action space $\mathcal{X}\times\mathcal{A}$.
Then the transition kernel becomes
\begin{equation}
    P(\theta, (s, a), (s', a)) = P(s'|s, a)\pi_{\theta}(a'|s'),
\end{equation}
and assuming the interchange between the derivative and the integration,
\begin{equation}
    \frac{d}{d\theta}P(\theta, (s, a), (s', a)) = P(s'|s, a)\frac{d}{d\theta}\pi_{\theta}(a'|s').
\end{equation}
Consequently, when the objective involves optimizing a long-run average or stationary expectation of a cost or performance metric of the MDP, our gradient estimator offers a straightforward implementation.
Its computational tractability makes it particularly suitable for gradient-based optimization algorithms.

\subsection{Dynamics Models parametrized by Neural Networks}
An alternative scenario arises when the transition dynamics of the system are inferred directly from observed data, an approach commonly referred to as model-based RL. 
Given a learned dynamics model, one typically simulates trajectories from this model to facilitate policy optimization.
However, since optimization occurs based on the approximated model rather than the true dynamics, assessing the sensitivity of the learned policy's performance with respect to changes in the estimated dynamics is crucial.
This question of sensitivity analysis naturally aligns with the capabilities of our proposed gradient estimator. 
Next we discuss some details on how system dynamics can be parametrized.
The easiest way to parametrize the dynamics is through a NN that is a deterministic mapping from a state-action pair to another state.
A more general and sophisticated parametrization is a probabilistic one where the NN maps a state-action pair to parameters of a parametrized distribution.
Empirically, probabilistic models have demonstrated performance comparable to model-free RL methods \cite{chua2018deep}.
Let $\mathcal{D}$ denote a training dataset that stores the transitions that has been experienced, then one can minimize the negative log prediction density (NLPD) as our loss function:
\begin{equation}
    loss(\phi) := - \sum_{(s_t, a_t, s_{t+1}) \in \mathcal{D}} \log P_{\phi}\left(s_{t+1}|s_t, a_t\right).
\end{equation}
where $P_{\phi}\left(s_{t+1}|s_t, a_t\right)$ is the transition density under NN parameter $\phi$.
Let $\phi^* := argmin_{\phi} loss(\phi)$ be the learned dynamics model parameter, then we can simulate the MDP with the learned transition dynamics $s_{t+1}\sim P_\phi\left(\cdot |s_t, a_t\right)$.
In this case, let $\pi^*$ be an optimized policy, then
\begin{equation}
    P(\phi, (s, a), (s', a')) = P_\phi\left(s' |s, a\right)\pi^*(a'|s'),
\end{equation}
and assuming the interchange between the derivative and the integration
\begin{equation}
    \frac{d}{d\phi}P(\phi, (s, a), (s', a')) = \frac{d}{d\phi}P_\phi\left(s' |s, a\right)\pi^*(a'|s').
\end{equation}
A more specific example where the NN maps to a Gaussian distribution can be found in \cite{chua2018deep}.
The transition densities and their derivatives can be computed from forward-feeding and back-propagations of NNs again. 
Therefore, without imposing any further restrictive conditions, our gradient estimator provides an effective tool for conducting sensitivity analysis on optimized policies $\pi^{*}$, specifically regarding perturbations in the learned dynamics model parameters.

\section{Numerical Experiments \jw{Ready to be reviewed}}\label{sec-NU}
In this section, we present numerical experiment results on three different scenarios: a heavy-traffic single server queue waiting time sequence where the state space is continuous; a multi-class queuing network control problem where the policy is parametrized by a neural network; and last an ising model control problem where the control is also parametrized by a neural network but with higher-dimensional state space.

\subsection{$M/M/1$ Queue}
We first illustrate our results in a simple $M/M/1$ customer waiting time sequence setting. 
Although our estimator has almost minimum requirement to implement, many of the previous gradient estimators that we want to compare with require analytical knowledge on the parametrized transition kernel or specific ways of parametrization and the $M/M/1$ example satisfies all the analytical requirements.
Nevertheless, the simple queuing model can still pose complex behavior and make the estimation task hard under heavy traffic scenarios because of the long regeneration cycles and the slow mixing rate.
 For the remainder of the section, $X_n$ represents the waiting time of the $n$th customer where $X_0 = 0$. 
The inter-arrival rate is fixed at 5 and the service rate is $\theta$. 
Thus our Markov chain evolves according to the Lindley recursion:
\begin{equation}
    X_{n+1} = \max(0, X_{n} + S^{\theta}_n - T_n),
\end{equation}
where $\newnota{S}{S^{\theta}_n}$'s are independent exponential random variables with rate $\theta$ and $\newnota{T}{T_n}$'s are independent exponential random variables with rate $5$.
The mean waiting time under stationarity is of interest here so in this case $\gamma(\theta) = \mathbb{E}_{X \sim \pi(\theta)}\left[X\right]$.
The goal is to estimate $\frac{d}{d\theta}\gamma(\theta)|_{\theta = \theta_0}$ where $\theta_0=5.2$ in our experiment.
The true analytical derivative here is rounded to about $-24.96$.
We then present several benchmark gradient estimators and the details on how they are implemented in the $M/M/1$ setting.

\subsubsection{The IPA estimator \cite{glasserman1993regenerative}}
Let $X(\theta) = \{X_n(\theta); n \geq 0\}$ be a Markov chain parametrized by $\theta$. 
Then under certain conditions,
\begin{equation}\label{eq: IPA-regen}
    \gamma'(\theta) = \mathbb{E}[X_\infty(\theta)]' = \mathbb{E}[X'_\infty(\theta)] = \frac{\mathbb{E}\Big[\sum_{i = \tau_{k-1}}^{\tau_k}X'_i(\theta)\Big]}{\mathbb{E}[\tau_k - \tau_{k-1}]},
\end{equation}
where $X'(\theta) := \{X'_n(\theta), n \geq 0\}$ represents the derivative sequence of the process $X$ and $\tau_k$ represents the k-th regeneration time of $X'$.
This paper mainly considers the special case where 
\begin{equation}
    X_{n+1}(\theta) = \phi(X_n(\theta), U_n(\theta)), \quad n\geq 0,
\end{equation}
for some recursive function $\phi$ and input sequence $U(\theta) = \{U_n(\theta); n \geq 0\}$.
Our $M/M/1$ setting fits exactly here. 
Note that by inverse CDF method, we know $-\frac{1}{\theta}\ln(U)$ follows the same distribution as $S^{\theta}_n$ where $U\sim Unif(0,1)$.
Thus in our case, the waiting time evolves according to
\begin{equation}
    X_{n+1} = \max(0, X_n -\frac{1}{\theta}\ln(U_n) - T_n),
\end{equation}
then 
\begin{align}
    X'_{n+1} &= X'_n + \frac{1}{\theta^2}\ln(U_n), \quad for \quad X_n(\theta) -\frac{1}{\theta}\ln(U_n) - T_n > 0,\\
    X'_{n+1} &= 0,  \quad\quad\quad\quad\quad\quad otherwise.
\end{align}
Thus the problem is then transformed to estimating the stationary mean of the derivative process $X'$ while $X'$ possesses the same regenerative structure as $X$ since both processes hit $0$ at the same moment.
Now define
\begin{equation}
    H^N_{IPA} := \frac{\frac{1}{N}\sum_{i=1}^N H^{(i)}}{\frac{1}{N}\sum_{i=1}^N \tau^{(i)}},
\end{equation}
where $\tau^{(i)}$'s are i.i.d. copies of the length of one regeneration cycle, and $H^{(i)}$'s are i.i.d. copies (independent from the $\tau^{(i)}$'s already generated) of the sum over one regeneration cycle defined as
\begin{equation}
    H^{(i)} := \sum_{j = \tau_{i-1}}^{\tau_i}X'_i.
\end{equation}
Hence based on \eqref{eq: IPA-regen}, $H^N_{IPA}$ is the proposed estimator for $\gamma'(\theta_0)$. 
Table~\ref{tab:performance} records the performance of $H^N_{IPA}$ with different values $N$. 
Note that the regeneration structure is utilized here and hence caused the initial biases for small $N$'s but it is asymptotically unbiased.

\subsubsection{The Phantom Estimator \cite{heidergott2006measure}} \label{Phantom}
Assume that for any $x \in \mathcal{X}$, $P'(\theta)$ exists such that
\begin{equation}
    \frac{d}{d \theta} \int_S P(\theta, x, d y) g(y)=\int_S P^{\prime}(\theta, x, d y) g(y)
\end{equation}
for any $x \in \mathcal{X}$. 
Then let $\left(\left[P^{\prime}(\theta)\right]^{+}(x ; \cdot),\left[P^{\prime}(\theta)\right]^{-}(x ; \cdot)\right)$ denote the Hahn-Jordan decomposition of the signed measure $P^{\prime}(\theta, x, \cdot)$ and define the normalizing constant to be
$$
c_{P(\theta)}(x):=\left[P^{\prime}(\theta)\right]^{+}(x ; \mathcal{X})=\left[P^{\prime}(\theta)\right]^{-}(x ; \mathcal{X}).
$$
Also define ${P}^+(\theta)$ and ${P}^-$ to be the normalized transition kernel corresponding to the Hahn-Jordan decomposition:
\begin{equation}
    P^{+}(\theta, x, \cdot)=\frac{\left[P^{\prime}(\theta)\right]^{+}(x ; \cdot)}{c_{P(\theta)}(x)}, \quad P^{-}(\theta, x, \cdot)=\frac{\left[P^{\prime}(\theta)\right]^{-}(x ; \cdot)}{c_{P(\theta)}(x)}.
\end{equation}
In the $M/M/1$ setting, an easy decomposition is to decompose the exponential service time $Exponential(\theta)$ to be a difference between an $Exponential(\theta)$ r.v. and a $Gamma(2, \theta)$ r.v. and the normalizing constant $c_{P(\theta)}(x) = 1/\theta$ for all $x$.
For decompositions of more common distributions, refer to Table A.1 in \cite{farenhorst2010efficient}.
Next, we introduce the so called phantom processes $X^+:=\{X_n^+, n \geq 0\}$ and $X^-:=\{X_n^-, n \geq 0\}$ where $X_0^+=X_0^-=x$ for some starting point $x$, the first transitions from $X^+_0$ to $X^+_1$ and from $X^-_0$ to $X^-_1$ are governed by kernels ${P}^+(\theta)$ and ${P}^-(\theta)$ respectively.
After the first transition, for $n \geq 2$, $X^+_n$ and $X^-_n$ evolve according to the regular coupled kernel $\Bar{P}(\theta)$.
Then under some ergodicity conditions, our desired quantity can be written as
\begin{equation}\label{'eq: heidergott'}
    \gamma'(\theta) = \mathbb{E}_{X_0 \sim \pi(\theta)}\left[\mathbf{D}\left(P(\theta), g ; X_0\right)\right],
\end{equation}
where
\begin{equation}\label{'eq: MVD'}
    \mathbf{D}\left(P(\theta), g ; x\right)=\mathbb{E}_{x}\left[c_{P(\theta)}(x) \sum_{n=0}^{\tau^{\pm}_{x}}\left(g(X^+_n)-g\left(X^-_n\right)\right)\right],
\end{equation}
and $\tau^{ \pm}=\inf \left\{n \in \mathbb{N}: X_n^{+}=X_n^{-}\right\}$ is the coupling time between $X^+$ and $X^-$.
Thus \eqref{'eq: heidergott'} transforms the derivative estimation into a stationary mean estimation for a different function.
In the single server queue setting, $0$ is an atom for the waiting time sequence and with that available, \cite{heidergott2006measure} proposes to utilize the following expression to estimate the derivative:
\begin{equation}
    \gamma'(\theta_{0}) = \frac{1}{\mathbb{E}\left[\tau^{0}_{\theta}\right]}\mathbb{E}_{X_0 = 0}\left[\sum_{i=0}^{\tau^{0}_{\theta}-1}\mathbf{D}\left(P(\theta), g ; X_i\right)\right]
\end{equation}
where $\tau^{0}_{\theta}:= \inf\{n \geq 1: X_n = 0\}$ is the first time the chain hits $0$.
Now define 
\begin{equation}\label{'eq: MVD-estimator'}
    \hat{\mathbf{D}}\left(P(\theta), g ; x\right)=\frac{1}{\theta} \sum_{n=0}^{\tau^{\pm}_{x}}\left(g(X^+_n)-g\left(X^-_n\right)\right),
\end{equation}
to be a estimator for $\mathbf{D}\left(P(\theta), g ; x\right)$.
Then assume we generate $N$ regeneration cycles.
Then the phantom estimator is defined as 
\begin{equation}
    H_{Ph}^{N}:=  \frac{\frac{1}{N}\sum_{i=1}^N \sum_{j=0}^{\tau^{0}_{\theta, (i)}-1}\hat{\mathbf{D}}\left(P(\theta), g ; X_j^{(i)}\right)}{\frac{1}{N}\sum_{i=1}^N \tau^{0}_{\theta, (i)}},
\end{equation}
where $\tau^{0}_{\theta, (i)}$ is the length of the $i$-th regeneration cycle, $X_{j}^{(i)}$ is the $j$-th state in the $i$-th regeneration cycle and each of the estimator $\hat{\mathbf{D}}$ is generated independently.  
Table~\ref{tab:performance-ph} records the performance of $H^{N}_{Ph}$ with different values of $N$. 
This estimator is also asymptotically unbiased.
In \cite{heidergott2010perturbation}, the authors propose another implementation of the phantom estimator in the single server queue setting with a more complex dependence on the regeneration cycle and thus the performance is worse in the heavy traffic scenarios comparing to the implementation of the phantom estimator experimented here.

\subsubsection{The Regeneration Estimator \cite{glynn1995likelihood}}
Assume the Markov chain $X$ possesses a regenerative structure and thus the stationary mean can be written as 
\begin{equation}
    \gamma(\theta) = \mathbb{E}[f(X_\infty)] = \frac{u(\theta)}{l(\theta)},
\end{equation}
where
\begin{align}
    u(\theta) &:= \mathbb{E}_s^\theta\Big[\sum_{k=1}^{\tau_{\theta}} f(X_k)\Big];\\
    l(\theta) &:= \mathbb{E}_s^\theta[\tau_{\theta}]
\end{align}
for some regenerative state $s$ and regeneration time $\tau_{\theta} := \inf\{n\geq 1: X_n = s\}$.
This paper shows that under appropriate conditions,
\begin{equation}\label{'eq: glynn-lecuyer'}
    \gamma'(\theta) = \frac{u'(\theta)l(\theta) - l'(\theta)u(\theta)}{l(\theta)^2}
\end{equation}
where
\begin{align}
    u'(\theta) &:= \mathbb{E}_s^\theta\left[\left(\sum_{k=1}^{\tau_{\theta}}p'(\theta, X_{k-1}, X_{k})\right)\left(\sum_{k=1}^{\tau_{\theta}} f(X_k)\right)\right];\\
    l'(\theta) &:= \mathbb{E}_s^\theta\left[\left(\sum_{k=1}^{\tau_{\theta}}p'(\theta, X_{k-1}, X_{k})\right)\tau_{\theta}\right],
\end{align}
where $p'$ is defined as in Assumption \ref{assumption: r-g-A4}.
Thus \eqref{'eq: glynn-lecuyer'} is a nonlinear function of four expectations, each of which can be estimated by generating $N$ regenerative cycles and averaging the corresponding quantities.
Now we define the estimators for each of those terms:
\begin{equation}
    \begin{aligned}
        \hat{u}_{N}(\theta) &:= \frac{1}{N}\sum_{i=1}^{N}\left(\sum_{k=1}^{\tau^{0}_{\theta, (i)}}f(X_k^{(i)})\right);\\
        \hat{u'}_{N}(\theta) &:= \frac{1}{N}\sum_{i=1}^N \left(\sum_{k=1}^{\tau^{0}_{\theta, (i)}}p'(\theta, X_{k-1}^{(i)}, X_{k}^{(i)})\right)\left(\sum_{k=1}^{\tau^{0}_{\theta, (i)}}f(X_k^{(i)})\right);\\
        \hat{l}_{N}(\theta) &:= \frac{1}{N}\sum_{i=1}^{N}\tau^{0}_{\theta, (i)};\\
        \hat{l'}_{N}(\theta) &:= \frac{1}{N}\sum_{i=1}^N \left(\sum_{k=1}^{\tau^{0}_{\theta, (i)}}p'(\theta, X_{k-1}^{(i)}, X_{k}^{(i)})\right)\tau^{0}_{\theta, (i)},
    \end{aligned}
\end{equation}
where $\tau^{0}_{\theta, (i)}$ is the length of the $i$-th regeneration cycle with the regeneration state being $0$, $X_{k}^{(i)}$ is the $k$-th state in the $i$-th regeneration cycle.
Thus the final estimator is of the form
\begin{equation}
    H_{Re}^{N}:= \frac{\hat{u'}_{N}(\theta)\hat{l}_{N}(\theta) - \hat{l'}_{N}(\theta)\hat{u}_{N}(\theta)}{\hat{l}_{N}(\theta)^2}.
\end{equation}
The estimator is biased but consistent as $N \rightarrow \infty$.
The performance is recorded in Table~\ref{tab:performance-lr}\\

\subsubsection{The Likelihood Ratio Estimator \cite{glynn2019likelihood}}
This work proposes two likelihood ratio based estimators that do not depend on the availability of the regeneration structure.
We implement the one with a better asymptotic variance \cite{glynn2019likelihood}. 
Assume the stationary mean $\gamma(\theta)$ is known, then the proposed estimator takes the form
\begin{equation}
    H_{LR}^{N} := \frac{1}{N}\sum_{k=1}^{N-1} \sum_{l=k}^{N-1} \left(f(X_l) - \gamma(\theta_0)\right)p'(\theta_0, X_{k-1}, X_{k}),
\end{equation}
where $p'$ is defined as in Assumption \ref{assumption: r-g-A4}.
The performance is
is recorded in Table~\ref{tab:performance-LR}.
Note that unlike the previous estimators, as $N$ gets larger, the variance of $H^N_{LR}$ does not go to 0.
Thus the work-variance product grows without bound as we try make the bias go to 0. \\

\subsubsection{Our estimator $H_4^{k, m, L}(X, Y)$ and comparison with benchmark methods}

We implement our proposed estimator $H_4^{k, m, L}(X, Y)$ with parameter selections $k = L$ and $m = 100$. 
Table~\ref{tab:performance-h4} summarizes the numerical performance across various choices of the parameter $L$.
The results highlight a rapid decrease in inefficiency, quantified by the work-variance product, as $L$ initially increases from smaller values.
This improvement rate slows down significantly beyond an optimal region around $L \approx 1000$. 
This empirical behavior aligns with our theoretical guideline, suggesting it is generally advantageous to err on the side of selecting a somewhat larger $L$ rather than choosing a smaller, potentially suboptimal value.

Figure~\ref{fig:Comparisons} illustrates the comparative work-variance performance of our unbiased derivative estimator against several benchmark methods previously discussed. 
Specifically, the benchmarks include the Infinitesimal Perturbation Analysis (IPA) estimator $H_{\mathrm{IPA}}^{1\times 10^5}$, the regeneration estimator $H_{\mathrm{Re}}^{4\times 10^4}$, the phantom estimator $H_{\mathrm{Ph}}^{3\times 10^3}$, and the likelihood ratio estimator $H_{\mathrm{LR}}^{1\times 10^5}$.
To ensure a fair and representative comparison, the benchmark methods are evaluated at their respective optimal parameter choices, providing their best achievable performance as a reference.

Notably, our estimator demonstrates superior or competitive performance relative to all benchmark methods, with the exception of the IPA estimator, to which it remains very close in performance.
Our estimator's optimal choice of $L$ significantly reduces the work-variance product, showcasing its effectiveness and robustness.
Overall, the proposed estimator $H_4^{k, m, L}$ provides substantial efficiency advantages in estimating derivatives of stationary means, especially when $L$ is chosen according to our practical guidelines.

\begin{figure}[htbp]
    \centering
    \includegraphics[height=2.40in]{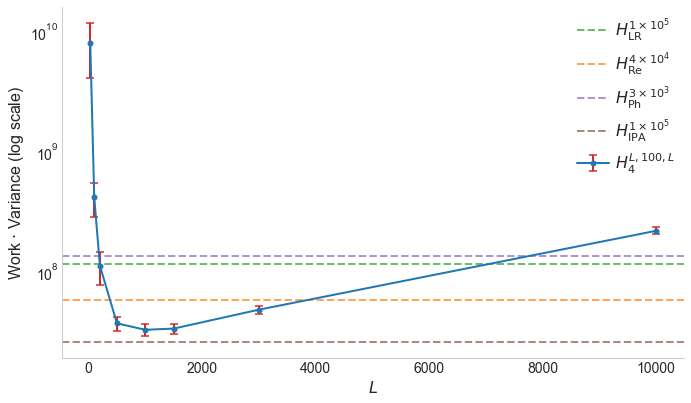}
    \caption{Comparison of the work-variance product of our unbiased derivative estimator $H_{4}^{L,100,L}$ with various previous methods as the parameter $L$ varies. The blue line with error bars illustrates our estimator's performance, depicting the trade-off between bias correction effort and variance reduction. The horizontal dashed lines represent the optimized work-variance products achieved by benchmark methods: the likelihood ratio estimator $H_{\mathrm{LR}}^{1\times10^5}$ (green), the regeneration estimator $H_{\mathrm{Re}}^{4\times10^4}$ (orange), the phantom estimator $H_{\mathrm{Ph}}^{3\times10^3}$ (purple), and the infinitesimal perturbation analysis estimator $H_{\mathrm{IPA}}^{1\times10^5}$ (brown). Error bars represent the 90\% confidence intervals based on simulation results.}
    \label{fig:Comparisons}
\end{figure}

\subsection{Multi-urn Ehrenfest Model}

In our second numerical illustration, we study a multi-urn Ehrenfest model \cite{karlin1965ehrenfest}, which characterizes the evolution of $n$ gas particles distributed across $N$ interconnected urns. 
Such models frequently arise in the analysis of particle diffusion processes, with control and sensitivity analyses playing critical roles in chemical and physical applications.
The model's state at discrete time $t$ is represented by the vector $X_{t} = \left(x_{t,i}\right)_{i=1}^{N}$, where each component $x_{t,i}$ denotes the number of particles in urn $i$ at time $t$.

At each discrete step, a particle from the entire population is uniformly selected, and transitions from urn $i$ to urn $j$ with probability
\begin{equation}
p_{ij} =
\begin{cases}
\frac{\theta_{i}}{n}, & i \neq j,\\[8pt]
1 - \frac{\theta_{i}(n-1)}{n}, & i = j,
\end{cases}
\end{equation}
where $\theta = \left(\theta_{i}\right)_{i=1}^{N}$ characterizes the diffusion rate at which particles escape from urn $i$.
These diffusion parameters are often influenced by external factors, such as heating sources or magnetic fields.
The performance metric we investigate is defined as:
\begin{equation}
f(X_t) = \sum_{i=1}^{N}\mathbbm{1}{x_{t,i} \geq C_i},
\end{equation}
where $C_i$ are predetermined occupancy thresholds for each urn.
The function $f(X_t)$ thus measures how many urns have reached or exceeded their respective thresholds at time $t$.
We are interested in the long-run average, defined by:
\begin{equation}
\gamma(\theta) := \lim_{T \rightarrow \infty} \frac{1}{T}\sum_{t=0}^{T} f(X_t),
\end{equation}
which represents the expected number of urns meeting or surpassing the occupancy thresholds under the stationary distribution.

The optimization of the diffusion parameters $\theta$ requires accurate estimation of the gradient $\gamma'(\theta)$. However, applying certain standard gradient estimation methods in this context presents difficulties. 
The Infinitesimal Perturbation Analysis (IPA) estimator is either impossible or not straightforward to implement due to the discrete state transitions inherent in the Ehrenfest model. 
Additionally, the likelihood ratio (LR) estimator necessitates explicit knowledge of the stationary distribution in closed form, which is unavailable here. 
Consequently, we restrict our numerical comparisons to our proposed estimator, the regenerative estimator (Re), and the Phantom estimator (Ph).

Figure~\ref{fig:ehrenfest_comparison} displays the numerical comparison of these three methods.
It is evident from the plot that our estimator (denoted $H_4^{L,100,L}$) outperforms the regenerative estimator for appropriately selected batch sizes $L$.
We selected moderate values for $N$ and $n$ to ensure that the system could regenerate with relative ease. 
However, the performance of the regeneration estimator rapidly deteriorates as the dimension increases, particularly when an effective regeneration state is unavailable. 
We will discuss more about this issue in Section \ref{sec: ising}.
On the other hand, the Phantom estimator exhibits superior performance in this scenario.
This advantage arises primarily from the distinct coupling mechanisms employed by each estimator.
Although the implementation of the Jordan decomposition is typically challenging—particularly impractical for neural network parameterizations—in this specific problem, it can be straightforwardly applied.
As a result, the Phantom estimator generates coupled Markov chains that differ by at most a single particle transition, leading to rapid coupling.
In contrast, our proposed estimator couples arbitrary states to a fixed reference state, which may result in longer coupling times and, consequently, higher computational demands.

\begin{figure}[htbp]
\centering
\includegraphics[width=0.7\linewidth]{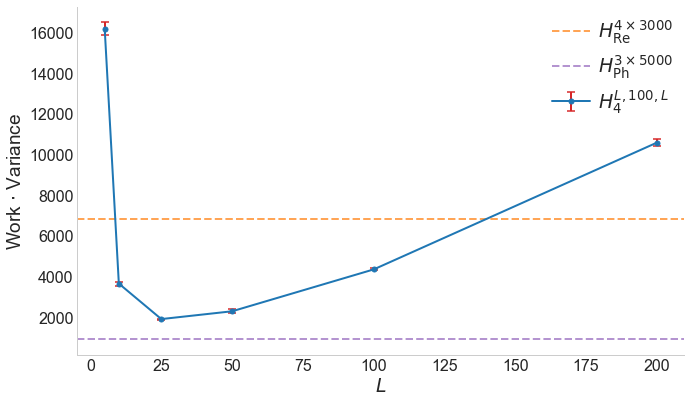}
\caption{Comparison of the work-variance product for the multi-urn Ehrenfest model across different estimators: our proposed estimator ($H_4^{L,100,L}$), the regenerative estimator ($H_{\mathrm{Re}}^{4\times 3000}$), and the Phantom estimator ($H_{\mathrm{Ph}}^{3\times 5000}$). Error bars indicate $90\%$ confidence intervals obtained from simulation experiments. The model consists of 5 interconnected urns containing a total of 25 particles, with identical diffusion rates $\theta_i = 1$ and occupancy thresholds $C_i = 6$ for all urns.}
\label{fig:ehrenfest_comparison}
\end{figure}

\jw{Add discussion and numerical plots here later}

\subsection{Two-class queuing network control}
Next we present a parallel server queuing network control problem as a Markov decision process where the policy is parametrized by a neural network.
Figure \ref{fig:N-model} shows the queuing network called the N-model with two classes of jobs whose arriving in the system with their inter-arrival time being exponentially distributed and with rates $\lambda_1 = 1.3\rho$ and $\lambda_2 = 0.4\rho$ where $\rho=0.9$ is the traffic intensity parameter.
Buffer 1 and Buffer 2 are buffers that hold class 1 and class 2 jobs respectively. 
Server 1 and Server 2 are two servers in the system.
Server 1 can process only class 1 jobs with a exponential service rate $m_1 = 1$ and Server 2 can process class 1 jobs with a exponential service rate $m_2 = 2$ and can process class 2 jobs with a exponential service rate $m_3 = 1$.
The state of the Markov decision process is a two dimensional vector $(x_1, x_2)$ where $x_1$ is the number of jobs in Buffer 1 and $x_2$ is the number of jobs in Buffer 2.
The holding cost of a state $(x_1, x_2)$ is defined as $h((x_1, x_2)) := 10x_1 + x_2$.
Let $\pi_{\bm{\theta}}$ denote the neural network parametrized policy that maps the state $(x_1, x_2)$ to a probability distribution over a binary action space where $\bm{\theta}$ is the weights and biases in the neural network.
An action $a \in \{1, 2\}$ is then sampled from the distribution $\pi_{\bm{\theta}}(a|x_1, x_2)$ whose gradient $\nabla_{\bm{\theta}} \pi_{\bm{\theta}}(a|x_1, x_2)$ can be numerically computed by back-propagation.
If $a = 1$, then Server 2 would preemptive class 1 jobs and if $a = 2$, Server 2 would preemptive class 2 jobs.
Since actions are finite, we can compute the exact transition probabilities and their gradients by conditioning on the actions we take and summing them up.
Check Section 5.3 in \cite{dai2022queueing} for the full transition dynamics of the discrete time Markov chain model of this queuing network.
In the reinforcement learning setting, we want to minimize the stationary mean holding cost of the network denoted as $\gamma(\bm{\theta})$.
Thus the goal here is to estimate the gradient of the weights and biases of the policy neural network with respect to the stationary mean holding cost incurred by the system, i.e. $\nabla_{\bm{\theta}}\gamma(\bm{\theta})$. 
Here we implement our estimator and the Glynn-L'ecuyer regeneration estimator as they are the only eligible options. 
To the best of our knowledge, deriving an IPA estimator or performing a Jordan decomposition for transition kernels parameterized by non-trivial neural networks remains impractical.
Furthermore, we do not have exact knowledge of the stationary mean so the likelihood ratio based method in \cite{glynn2019likelihood} is also not applicable here.
The entire gradient can be estimated but for the purpose of comparison and visualization, we only report the estimation of the partial derivative with respect to the last bias term of the neural network. 
Figure \ref{fig:N-model-Comparisons} presents the performance comparison between our method and the regeneration method and it shows that at optimal choice of $L$, the work-variance product decreases almost 50\% and for a wide range of choices for $L$, the performance of our estimator triumphs the regeneration estimator.

\begin{figure}[htbp]
    \centering
    \begin{subfigure}[b]{0.39\textwidth}
        \centering
        \includegraphics[width=\textwidth]{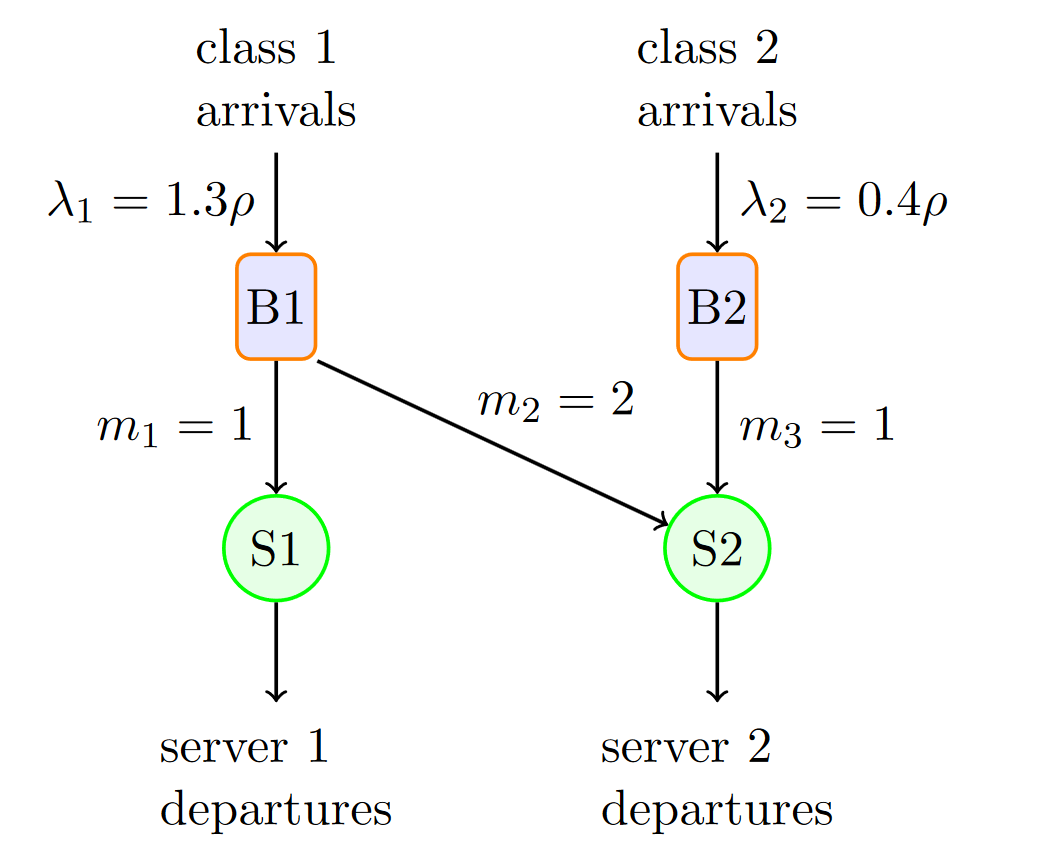}
        \caption{}
        \label{fig:N-model}
    \end{subfigure}
    \hfill
    \begin{subfigure}[b]{0.54\textwidth}
        \centering
        \includegraphics[width=\textwidth]{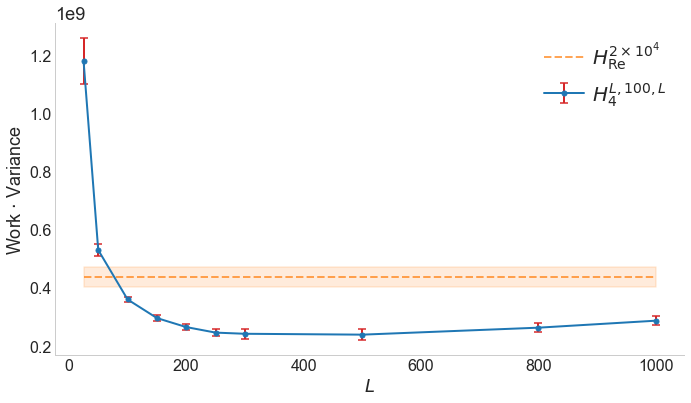}
        \caption{}
        \label{fig:N-model-Comparisons}
    \end{subfigure}

    \caption{(a) Structure of the two-class queuing network (N-model) with two job classes and two servers. 
    (b) Performance comparison between our proposed gradient estimator $H_4^{L,100,L}$ and the regeneration estimator $H_{Re}^{2\times 10^4}$. The shaded orange band represents the 90\% confidence interval for the regeneration estimator. At optimal choices of $L$, our estimator achieves approximately a 50\% reduction in the work-variance product, and consistently outperforms the regeneration estimator across a broad range of $L$ values.}
    \label{fig:N-model-combined}
\end{figure}

\subsection{Ising model Control}\label{sec: ising}
In the $M/M/1$ example, we show that our estimator outperforms the almost all previous methods except for the IPA estimator for a wide range choice of the skipping parameter $L$.
In the 2-class queuing network control problem, our estimator also exhibit the similar trait comparing to the Glynn-L'ecuyer regeneration based likelihood ratio estimator when the parametrization is by neural networks. 
In this section, we consider a non-queuing setting where the regeneration state becomes less accessible as dimension of the problem scales.\\
There has been some recent research that aims to maximize influence in social networks that are modeled by ising models \cite{liu2010influence,lynn2018maximizing,lynn2016maximizing}.
Ising models can represents general social networks with different interaction frameworks but here for the illustrative purposes we adapt the simple $2$-dimensional $d \times d$ square lattice as our example where each node only connects to its adjacent neighbors.
There is in total $d^2$ nodes and node on the $i$-th row and $j$-th column is indexed as $(i, j)$.
Each node $(i, j)$ is assigned a binary value $\sigma_{(i, j)} \in \{-1, 1\}$.
In the social network setting, $\sigma_i$ is the current opinion of individual $i$ such as being prone to republican or democratic party and each node holds some influence to its neighbors. 
The control element of this model is introduced by a external magnetic field. 
We represent the external field by a neural network $h^{\bm{\theta}}$ that maps the node index $(i, j)$ to a real number where $\bm{\theta}$ denotes the weights and biases of the neural network network. 
This external magnetic field is usually used to model the effect of a campaign or advertisement over the entire population's opinion formation.
To introduce the time element, we let $\sigma_{(i, j)}(t)$ to be the opinion of individual $(i, j)$ at time step $t$ and let the vector $\boldsymbol{\sigma}(t) = \{\sigma_{(i, j)}(t)\}_{i,j \in \{1,2,\dots,d\}^2}$ denote the current opinions of the entire network at time step $t$.
The goal in this case is to maximize the total opinion in the long run or the stationary expected opinion, $\gamma(\bm{\theta}) := \lim_{N \rightarrow \infty} \frac{1}{N}\sum_{t=0}^{N-1}M(\boldsymbol{\sigma}(t))$ where $M(\boldsymbol{\sigma}(t)) = \sum_{i,j \in \{1,2,\dots,d\}^2}\sigma_{(i, j)}(t)$ is the total opinion or the magnetization of the lattice at time $t$. 
The transition of the system follows the Glauber dynamics \cite{lynn2016maximizing}: at each time $t$, we sample an individual $(i, j)$ uniformly from the entire population and update its state according to
\begin{equation}
    P\left(\sigma_{(i, j)}(t+1)=1 \mid \sigma(t)\right)=\frac{e^{\beta\left(\sum_{(k, l) \in \mathcal{N}_{(i,j)}}\sigma_{(k, l)}(t)+h^{\bm{\theta}}(i, j)\right)}}{e^{-\beta\left(\sum_{(k, l) \in \mathcal{N}_{(i,j)}}\sigma_{(k, l)}(t)+h^{\bm{\theta}}(i, j)\right)} + e^{\beta\left(\sum_{(k, l) \in \mathcal{N}_{(i,j)}}\sigma_{(k, l)}(t)+h^{\bm{\theta}}(i, j)\right)}},
\end{equation}
where $\mathcal{N}_{(i,j)}$ is the set of nodes adjacent to node $(i, j)$ and $\beta$ is the inverse temperature which reflects the overall interaction strength.
The most intuitive way to construct a coupling kernel is by sampling the individuals with the same indices in both lattices and sampling their opinions following the maximum coupling distribution.
Our proposed gradient estimator provides a practical method for estimating $\nabla_{\bm{\theta}} \gamma (\bm{\theta})$ within numerical optimization frameworks aimed at identifying optimal targets for advertisement placement, thus maximizing influence subject to budgetary constraints.
Although our method is capable of estimating the full gradient, we report numerical results specifically for the partial derivative with respect to the final bias term of the neural network, consistent with the approach used in the queuing network control example. 
To evaluate performance, we compare our estimator against the Glynn-L’Ecuyer regeneration estimator across lattice dimensions $d = 2, 3, 4$.
For the regeneration estimator, we selected the regeneration state where $sigma_{(i, j)}=1$ for all indexes $(i, j)$, as alternative regeneration states yielded inferior performance in our trials.
Figure~\ref{fig:ising-lattices} presents the work-variance products of our estimator for varying values of the lagging parameter $L$, in comparison to the Glynn-L’Ecuyer estimator across the dimensions tested. 
The results clearly demonstrate that, as anticipated, our estimator increasingly outperforms the regeneration estimator as $d$ grows, primarily because regenerations become exceedingly rare events in higher-dimensional settings, substantially elevating the regeneration estimator’s work-variance. Consequently, for larger social networks, the performance difference between the two estimators transitions from quantitative to qualitative significance.

\begin{figure}[htbp]
    \centering
    \begin{subfigure}[b]{0.32\textwidth}
        \centering
        \includegraphics[width=\textwidth]{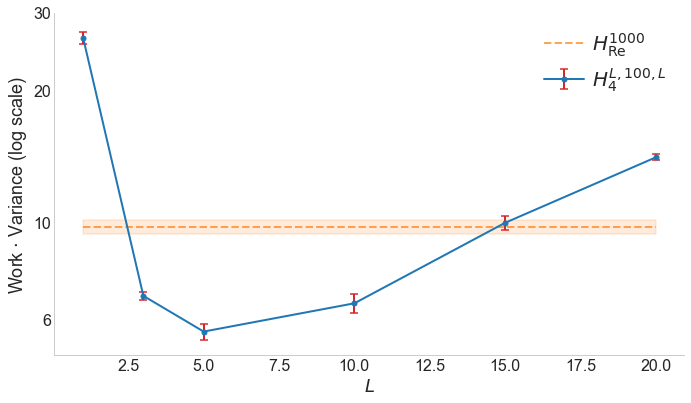}
        \caption{(a)}
        \label{fig:ising-l2}
    \end{subfigure}
    \hfill
    \begin{subfigure}[b]{0.32\textwidth}
        \centering
        \includegraphics[width=\textwidth]{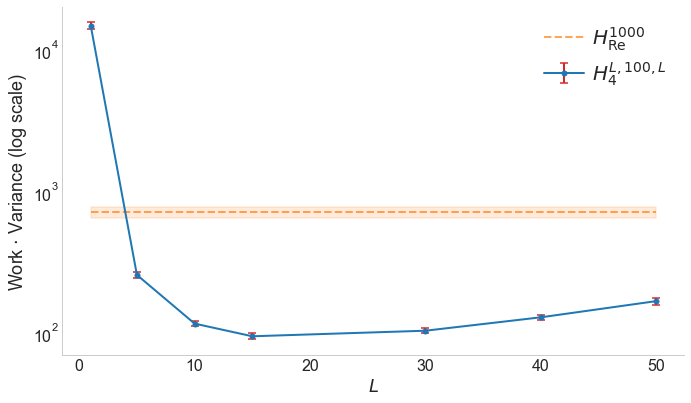}
        \caption{(b)}
        \label{fig:ising-l3}
    \end{subfigure}
    \hfill
    \begin{subfigure}[b]{0.32\textwidth}
        \centering
        \includegraphics[width=\textwidth]{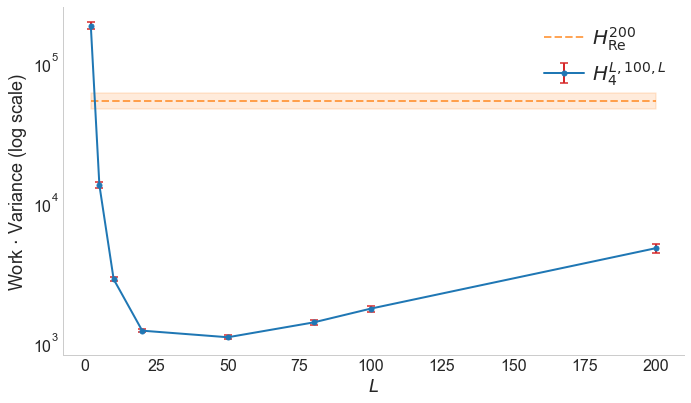}
        \caption{(c)}
        \label{fig:ising-l4}
    \end{subfigure}

    \caption{Work-Variance analysis for (a) $2 \times 2$, (b) $3 \times 3$, and (c) $4 \times 4$ Ising lattices. Each plot shows the comparison between our proposed estimator $H_4^{L,100,L}$ and the Glynn-L'Ecuyer regeneration estimator ($H_{Re}^N$) across different lattice dimensions. The shaded orange band represents the 90\% confidence interval for the regeneration estimator. Our estimator increasingly outperforms the regeneration estimator as the lattice dimension increases, demonstrating significant reductions in the Work-Variance product, especially evident in higher-dimensional lattices.}
    \label{fig:ising-lattices}
\end{figure}

\section{Proofs \jw{Ready to be reviewed}}
\subsection{Proof of Theorem \ref{theorem: L-representation}}\label{sec-proof-thm-L-representation}
\begin{proof}
    First see that for any $i \geq 0$,
    \begin{equation}\label{eq: thm-rep-1}
        \gamma'(\theta_0) = \mathbb{E}^{\theta_0}_{\pi(\theta_0)}\left[p'(\theta_0, X_{i}, X_{i+1})g(X_{i+1})\right],
    \end{equation}
    since the chain starts from stationarity.
    Next we define some notations.
    Define the fundamental solution to the $L$-skeleton chain Poisson's equation to be
    \begin{equation}
        \newnota{g-fu-L}{g^{fu, L}(x)}:= \mathbb{E}_{x}^{\theta_0}\left[\sum_{j=0}^{\infty}\bar{f}(X_{jL})\right],
    \end{equation}
    where $\newnota{bar-f}{\bar{f}(x)} = f(x) - \pi(\theta_0)f$. Note that $g^{fu, L}$
    Then we utilize \eqref{eq: thm-rep-1} to obtain the following equality,
    \begin{equation}
        \begin{aligned}
            L\gamma'(\theta_0) &= \sum_{i=0}^{L-1}\mathbb{E}_{\pi(\theta_0)}\left[p'(\theta_0, X_i, X_{i+1})g^{fu, 1}(X_{i+1})\right]\\
            &= \mathbb{E}_{\pi(\theta_0)}\left[\sum_{i=0}^{L-1}\left(p'(\theta_0, X_i, X_{i+1})Lg^{fu, L}(X_{L})\right)\right]\\
            &\quad + \sum_{i=0}^{L-1}\mathbb{E}_{\pi(\theta_0)}\left[p'(\theta_0, X_i, X_{i+1})\left(g^{fu, 1}(X_{i+1}) - Lg^{fu, L}(X_L)\right)\right].
        \end{aligned}
    \end{equation}
    The only remaining step is to show that the second term equals to $0$.
    Again since the chain starts from stationarity,
    \begin{equation}\label{eq: thm-rep-2}
        \mathbb{E}_{\pi(\theta_0)}\left[p'(\theta_0, X_i, X_{i+1})\left(g^{fu, 1}(X_{i+1}) - Lg^{fu, L}(X_L)\right)\right] = \mathbb{E}_{\pi(\theta_0)}\left[p'(\theta_0, X_0, X_{1})\left(g^{fu, 1}(X_{1}) - Lg^{fu, L}(X_{L-i})\right)\right].
    \end{equation}
    Therefore, 
    \begin{equation}
        \begin{aligned}
            &\sum_{i=0}^{L-1}\mathbb{E}_{\pi(\theta_0)}\left[p'(\theta_0, X_i, X_{i+1})\left(g^{fu, 1}(X_{i+1}) - Lg^{fu, L}(X_L)\right)\right]\\
            &= \mathbb{E}_{\pi(\theta_0)}\left[p'(\theta_0, X_0, X_{1})\left(Lg^{fu, 1}(X_{1}) - \sum_{i=0}^{L-1}Lg^{fu, L}(X_{L-i})\right)\right]\\
            &=\mathbb{E}_{\pi(\theta_0)}\left[p'(\theta_0, X_0, X_{1})L\mathbb{E}^{\theta_0}_{X_1 = X_1}\left[\left(g^{fu, 1}(X_{1}) - \sum_{i=0}^{L-1}g^{fu, L}(X_{L-i})\right)\right]\right]\\
            &= 0
        \end{aligned}
    \end{equation}
    where the first equality is due to \eqref{eq: thm-rep-2}, the second equality is from the tower property by conditioning on $X_1$ and the final equality holds by realizing the following equality:
    \begin{equation}
        \mathbb{E}_{x}^{\theta_0}\left[\sum_{i=0}^{L-1}g^{fu, L}(X_i)\right] = \mathbb{E}_{x}^{\theta_0}\left[\sum_{i=0}^{L-1}\sum_{j=0}^{\infty}\bar{f}(X_{i + jL})\right] = \mathbb{E}_{x}^{\theta_0}\left[\sum_{j=0}^{\infty}\bar{f}(X_{j})\right] = g^{fu, 1}(x).
    \end{equation}
    Hence it completes the proof.
\end{proof}

\subsection{Proof of Lemma \ref{lemma: Gamma}}\label{sec-proof1}
\begin{proof}
    Applying H\"older's inequality and get
    \begin{align}
        \Gamma^{L}_{\zeta}(x) \leq \mathbb{E}_{x}^{\theta_0}\left[G^L_z(X_L)^{2\zeta p}\right]^{\frac{1}{p}}\mathbb{E}_{x}^{\theta_0}\left[\left(\sum_{i=0}^{L-1}p'(\theta_0, X_i, X_{i+1})\right)^{\frac{2\zeta p}{p-1}}\right]^{\frac{p-1}{p}}. \label{eq: breakdown}
    \end{align}
    First we obtain an upper bound for the second term of the right-hand side. 
    Recall that we defined 
    \begin{equation}\label{eq: Mn-proof}
        M_n := \sum_{i=0}^{L-1}p'(\theta_0, x_{i}, x_{i+1})
    \end{equation}
    with $M_0$ = 0. 
    We now show that $\{M_n\}_{n=1,2,\dots}$ is a martingale.
    For fixed states $x, y$ and for all $h \leq \epsilon$,
    \begin{equation}\label{eq: omega-xy}
        \frac{\left|p(\theta_0 + h, x, y) - p(\theta_0, x, y) \right|}{h} \leq \omega_\epsilon(x, y)
    \end{equation}
    and 
    \begin{align}
        \int_{\mathcal{X}}\omega_\epsilon(x, y)P(\theta_0, x, dy) \leq V(x)\left|\Omega(x)\right|_{V} < \infty,
    \end{align}
    where 
    \begin{equation}
        \newnota{Omega-2}{\Omega(x)} = \int_{\mathcal{X}}\omega_{\epsilon}(x, y)P(\theta_0, x, dy),
    \end{equation}
    and the second inequality holds since $\left|\Omega\right|_{V} < \infty$ is implied by $\left|\Omega^{2q}\right|_{V} < \infty$ in Assumption \ref{assumption-moment}.
    Thus
    \begin{align}
        \int_{\mathcal{X}}p'(\theta_0, x, y)P(\theta_0, x, dy) &=\int_{\mathcal{X}}\lim_{h \rightarrow 0} \frac{p(\theta_0 + h, x, y) - p(\theta_0, x, y) }{h}P(\theta_0, x, dy) \\
        &=\lim_{h \rightarrow 0} \int_{\mathcal{X}}\frac{p(\theta_0 + h, x, y) - p(\theta_0, x, y)}{h}P(\theta_0, x, dy)\\
        &=\lim_{h \rightarrow 0} \frac{1}{h} \left(\int_{\mathcal{X}}p(\theta_0 + h, x, y)P(\theta_0, x, dy) - \int_{\mathcal{X}}p(\theta_0, x, y)P(\theta_0, x, dy)\right)\\
        &= 0,
    \end{align}
    where the second equality follows from the dominated convergence theorem and the last equality follows from 
    \begin{equation}
        \int_{\mathcal{X}}p(\theta, x, y)P(\theta_0, x, dy) = \int_{\mathcal{X}}P(\theta, x, dy)= 1,
    \end{equation} 
    for all $\theta$.
    Then
    \begin{align}
        \mathbb{E}^{\theta_0}_{x}\left[p'(\theta_0, X_i, X_{i+1})\right] &= \mathbb{E}^{\theta_0}_{x}\left[\mathbb{E}^{\theta_0}\left[p'(\theta_0, X_i, X_{i+1})| X_i\right]\right] = 0,
    \end{align}
    and thus $M_n$ is a martingale. 
    Furthermore, define
    \begin{align}
        \gamma_{\frac{2\zeta p}{p-1},n} &:= \mathbb{E}_{x}^{\theta_0}\left[\left|p'(\theta_0, X_n, X_{n+1})\right|^{\frac{2\zeta p}{p-1}} \right],
    \end{align}
    and observe that 
    \begin{align}
        \gamma_{\frac{2\zeta p}{p-1},n} &= \mathbb{E}_{x}^{\theta_0}\left[\mathbb{E}^{\theta_0}\left[\left|p'(\theta_0, X_n, X_{n+1})\right|^{\frac{2\zeta p}{p-1}} | X_n\right]\right]\\
        &\leq \mathbb{E}_{x}^{\theta_0}\left[\mathbb{E}^{\theta_0}\left[\omega_{\epsilon}(X_n, X_{n+1})^{\frac{2\zeta p}{p-1}} | X_n\right]\right]\\
        &\leq |\Omega^{\frac{2\zeta p}{p-1}}|_{V} \mathbb{E}_{x}^{\theta_0}\left[V(X_n)\right]\\
        &\leq |\Omega^{\frac{2\zeta p}{p-1}}|_{V} \left(V(x) + \frac{b}{1-\lambda}\right)\\
        &= |\Omega^{\frac{2\zeta p}{p-1}}|_{V} \left(1 + \frac{b}{1-\lambda}\right)V(x),
    \end{align}
    where the first inequality holds because \eqref{eq: omega-xy}; the second inequality holds due to Assumption \ref{assumption-moment} with $\zeta \leq \kappa$; and the third inequality holds since \eqref{eq: V-bar}. 
    Thus by Theorem 1 in \cite{dharmadhikari1968bounds}, we obtain an upper bound on the second term of \eqref{eq: breakdown}:
    \begin{equation}\label{eq: bound-1}
        \mathbb{E}\left[|M_L|^{\frac{2\zeta p}{p-1}}\right]^{\frac{p-1}{p}} \leq L^{\zeta} \left(C_{\frac{2\zeta p}{p-1}}|\Omega^{\frac{2\zeta p}{p-1}}|_{V} \left(1 + \frac{b}{1-\lambda}\right)V(x)\right)^{\frac{p-1}{ p}}
    \end{equation}
    where $\newnota{C-kappa}{C_{l}} := \left[8(l-1) \max(1, 2^{l-3})\right]^{l}$.
    Now we look at the first term in \eqref{eq: breakdown}.
    We define $(X_{L+t}, Z_t)_{t\in \mathbb{N}}$ to be a coupled Markov chains that evolves according to $\Bar{P}(\theta_0)$ with $X_L \sim P^L(\theta_0, x, \cdot)$ and $Z_0 = z$. 
    Then see that
    \begin{equation}\label{eq: lemma-bound-gz}
    \begin{aligned}
        \mathbb{E}_{x}\left[G^L_z(X_L)^{2\zeta p}\right]^{\frac{1}{2\zeta p}}
        &= \mathbb{E}_{x, z}^{\theta_0}\left[\left(\sum_{t=0}^{\infty}\left(\left(f(X_{(t+1)L}) - f(Z_{tL})\right)\mathbbm{1}\left(X_{(t+1)L} \neq Z_{tL}\right)\right)^{2\zeta p}\right)\right]^{\frac{1
        }{2\zeta p}}\\
        &\leq \sum_{t=0}^{\infty}\Big(\mathbb{E}_{x, z}^{\theta_0}\left[f(X_{(t+1)L})^{2\zeta p}\mathbbm{1}\left(X_{(t+1)L} \neq Z_{tL}\right)\right]^{\frac{1}{2\zeta p}}\\
        &\quad +\mathbb{E}_{x, z}^{\theta_0}\left[f(Z_{tL})^{2\zeta p}\mathbbm{1}\left(X_{(t+1)L} \neq Z_{tL}\right)\right]^{\frac{1}{2\zeta p}}\Big)\\
        &\leq \sum_{t=0}^{\infty}\Big(\mathbb{E}_{x, z}^{\theta_0}\left[f(X_{(t+1)L})^{2\zeta p+\delta}\right]^{\frac{1}{2\zeta p+\delta}}\\
        &\quad +\mathbb{E}_{x, z}^{\theta_0}\left[f(Z_{tL})^{2\zeta p+\delta}\right]^{\frac{1}{2\zeta p+\delta}}\Big)P_{x}\left(\tau_{X_L, z} > tL\right)^{\frac{\delta}{2\zeta p(2\zeta p+\delta)}}\\
        &\leq \left|f^{2\zeta p+\delta}\right|_{V}^{\frac{1}{2\zeta p+\delta}}\left(\Bar{V}(x)^{\frac{1}{2\zeta p+\delta}}+\Bar{V}(z)^{\frac{1}{2\zeta p+\delta}}\right)\sum_{t=0}^{\infty}P_{x}\left(\tau_{X_L, z} > tL\right)^{\frac{\delta}{2\zeta p(2\zeta p+\delta)}}\\
        &\leq \left|f^{2\zeta p+\delta}\right|_{V}^{\frac{1}{2\zeta p+\delta}}\left(\Bar{V}(x)^{\frac{1}{2\zeta p+\delta}}+\Bar{V}(z)^{\frac{1}{2\zeta p+\delta}}\right)\\
        &\quad \cdot \left(1 + \left(M\Bar{V}(x)+MV(z)\right)^{\frac{\delta}{2\zeta p(2\zeta p+\delta)}}\frac{\left(\rho^{\frac{\delta}{2\zeta p(2\zeta p+\delta)}}\right)^L}{1-\left(\rho^{\frac{\delta}{2\zeta p(2\zeta p+\delta)}}\right)^L}\right)
    \end{aligned}
    \end{equation}
    where the first inequality is due to Minkowski's inequality; the second inequality is from H\"older's inequality; the third inequality holds because Assumption \ref{assumption-moment} and \ref{assumption: drift}, specifically
    \begin{equation}
        \mathbb{E}_{x, z}^{\theta_0}\left[f(X_{tL})^{2\zeta p + \delta}\right] \leq \left|f^{2\zeta p+\delta}\right|_{V}^{\frac{1}{2\zeta p+\delta}}\mathbb{E}_{x, z}^{\theta_0}\left[V(X_{tL})\right] \leq \left|f^{2\zeta p+\delta}\right|_{V}^{\frac{1}{2\zeta p+\delta}}\Bar{V}(x);
    \end{equation}
    and the forth inequality holds if we apply Assumption \ref{assumption-coupling} and obtain
    \begin{align}
        P_{x}\left(\tau_{X_L, z} > tL\right) &= \int_{\mathcal{X}} P\left(\tau_{y, z} > tL\right) P^L(\theta_0, x, dy)\\
        &\leq \int_{\mathcal{X}} M\left(V(y) + V(z)\right)\rho^{tL} P^L(\theta_0, x, dy)\\
        &\leq M\left(\Bar{V}(x) + V(z)\right)\rho^{tL}.
    \end{align}
    Therefore by raising both sides to the power $2\zeta$, we obtain
    \begin{align}
        \mathbb{E}_{x}\left[G^L_z(X_L)^{2\zeta p}\right]^{\frac{1}{p}} &\leq 2^{4\zeta - 2}\left|f^{2\zeta p+\delta}\right|_{V}^{\frac{2\zeta}{2\zeta p+\delta}}\left(\Bar{V}(x)^{\frac{2\zeta}{2\zeta p+\delta}}+\Bar{V}(z)^{\frac{2\zeta}{2\zeta p+\delta}}\right)\nonumber\\
        &\quad \cdot \left(1 + M^{\frac{2\zeta\delta}{2\zeta p(2\zeta p+\delta)}}\left(\Bar{V}(x)^{\frac{2\zeta\delta}{2\zeta p(2\zeta p+\delta)}}+V(z)^{\frac{2\zeta\delta}{2\zeta p(2\zeta p+\delta)}}\right)\left(\frac{\left(\rho^{\frac{\delta}{2\zeta p(2\zeta p+\delta)}}\right)^L}{1-\left(\rho^{\frac{\delta}{2\zeta p(2\zeta p+\delta)}}\right)^L}\right)^{2\zeta}\right).\label{eq: bound-2}
    \end{align}
    where we use the inequality $(x + y)^{2\zeta} \leq 2^{2\zeta - 1}(x^{2\zeta} + y^{2\zeta})$.
    To emphasize the dependence of the starting state $x$, we rewrite the upper bound to be
    \begin{align}
        \mathbb{E}_{x}\left[G^L_z(X_L)^{2\zeta p}\right]^{\frac{1}{p}} &\leq V(x)^{\frac{1}{p}}2^{4\zeta - 2}\left|f^{2\zeta p+\delta}\right|_{V}^{\frac{2\zeta}{2\zeta p+\delta}}\left(1 +\left(\frac{b}{1-\lambda}\right)^{\frac{2\zeta}{2\zeta p+\delta}}+\Bar{V}(z)^{\frac{2\zeta}{2\zeta p+\delta}}\right)\nonumber\\
        &\quad \cdot \left(1 + M^{\frac{2\zeta\delta}{2\zeta p(2\zeta p+\delta)}}\left(1 + \left(\frac{b}{1-\lambda}\right)^{\frac{2\zeta\delta}{2\zeta p(2\zeta p+\delta)}}+V(z)^{\frac{2\zeta\delta}{2\zeta p(2\zeta p+\delta)}}\right)\left(\frac{\left(\rho^{\frac{\delta}{2\zeta p(2\zeta p+\delta)}}\right)^L}{1-\left(\rho^{\frac{\delta}{2\zeta p(2\zeta p+\delta)}}\right)^L}\right)^{2\zeta}\right).
    \end{align}
    Thus combining it with the bound on the martingale term \eqref{eq: bound-1}, we obtain that
    \begin{equation}
        \Gamma_{\zeta}^L(x) \leq V(x)\left(L^{\zeta}A_{z} + L^{\zeta}\left(\frac{\left(\rho^{\frac{\delta}{2\zeta p(2\zeta p+\delta)}}\right)^L}{1-\left(\rho^{\frac{\delta}{2\zeta p(2\zeta p+\delta)}}\right)^L}\right)^{2\zeta}B_{z}\right):= V(x)U_{\zeta, z}(L)
    \end{equation}
    where
    \begin{align}
        A_z^{\zeta} &:= 2^{4\zeta - 2}\left|f^{2\zeta p+\delta}\right|_{V}^{\frac{2\zeta}{2\zeta p+\delta}}\left(C_{\frac{2\zeta p}{p-1}}|\Omega^{\frac{2\zeta p}{p-1}}|_{V} \left(1 + \frac{b}{1-\lambda}\right)\right)^{\frac{p-1}{ p}}\left(1 +\left(\frac{b}{1-\lambda}\right)^{\frac{2\zeta}{2\zeta p+\delta}}+\Bar{V}(z)^{\frac{2\zeta}{2\zeta p+\delta}}\right);\\
        B_z^{\zeta} &:= A_z^{\zeta} \cdot M^{\frac{2\zeta\delta}{2\zeta p(2\zeta p+\delta)}}\left(1 + \left(\frac{b}{1-\lambda}\right)^{\frac{2\zeta\delta}{2\zeta p(2\zeta p+\delta)}}+V(z)^{\frac{2\zeta\delta}{2\zeta p(2\zeta p+\delta)}}\right)
    \end{align}
    do not depend on the starting state $x$ or $L$.
    Therefore,
    \begin{equation}
        \frac{\Gamma_{\zeta}^L(x)}{V(x)} \leq U_{\zeta, z}(L)
    \end{equation}
    for all $x \in \mathcal{X}$ and hence $\left|\Gamma_{\zeta}^{L}\right|_V \leq U_{\zeta, z}(L) < \infty$.
\end{proof}

\subsection{Proof of Theorem \ref{theorem: H2}}\label{sec-proof2}
\begin{proof}
    Denote the starting state as $x$, first observe that
    \begin{equation}\label{eq: thm-bd-1}
    \begin{aligned}
        \mathbb{E}_{\mu_0}^{\theta_0}\left[H\left(\{X_i\}_{i=k+tL}^{k+(t+1)L}\right)^2 \mathbbm{1}\left(X_{k+tL} \neq Y_{k+(t-1)L}\right)\right]^{\frac{1}{2}} &\leq \mathbb{E}_{\mu_0}^{\theta_0}\left[H\left(\{X_i\}_{i=k+tL}^{k+(t+1)L}\right)^{2\kappa}\right]^{\frac{1}{2\kappa}}P_{x}\left(\tau^L > k + tL\right)^{\frac{\kappa - 1}{2\kappa}} \\
        &= \mathbb{E}_{\mu_0}^{\theta_0}\left[\Gamma^{L}_{\kappa}(X_{k+tL})\right]^{\frac{1}{2\kappa}}P_{x}\left(\tau^L > k + tL\right)^{\frac{\kappa - 1}{2\kappa}} \\
        &\leq \left|\Gamma^{L}_{\kappa}\right|_{V}^{\frac{1}{2\kappa}}\mathbb{E}_{\mu_0}^{\theta_0}\left[V(X_{k+tL})\right]^{\frac{1}{2\kappa}}P_{x}\left(\tau^L > k + tL\right)^{\frac{\kappa - 1}{2\kappa}} \\
        &\leq \left|\Gamma^{L}_{\kappa}\right|_{V}^{\frac{1}{2\kappa}}\mu_0\Bar{V}^{\frac{1}{2\kappa}}P_{x}\left(\tau^L > k + tL\right)^{\frac{\kappa - 1}{2\kappa}},
    \end{aligned}
    \end{equation}
    where the first inequality is due to H\"older's inequality; the second equality is by the tower property; the second inequality follows from Lemma \ref{lemma: Gamma} and the last inequality follows from \eqref{eq: V-bar}. 
    The argument holds similarly for chain $Y$.
    Next we bound the tail probability of the coupling time.
    \begin{equation}\label{eq: thm-bd-2}
    \begin{aligned}
        P_{x}\left(\tau^L > k + tL\right) &= \int_{\mathcal{X}}\int_{\mathcal{X}} P_{x}\left(\tau_{x, y} > k + (t - 1)L\right) P^L(\theta_0, x, dy)\mu_{0}(dx)\\
        &\leq \int_{\mathcal{X}}\int_{\mathcal{X}} M\left(V(x) + V(y)\right)\rho^{k + (t - 1)L} P^L(\theta_0, x, dy)\mu_{0}(dx)\\
        &\leq M\left(\mu_{0}V + \mu_{0}\Bar{V}\right)\rho^{k + (t - 1)L},
    \end{aligned}
    \end{equation}
    where the first inequality follows from \eqref{eq: A-tail}.
    Now we are equipped to show that $H_{2}^{k,L}(X, Y)$ has finite second moment.
    \begin{equation}
    \begin{aligned}
        \mathbb{E}_{\mu_0}^{\theta_0}\left[H_{2}^{k,L}(X, Y)^2\right]^{\frac{1}{2}} &\leq \mathbb{E}_{\mu_0}^{\theta_0}\left[H\left(\{X_i\}_{i=k}^{k+L}\right)^{2}\right]^{\frac{1}{2}}  + \sum_{t=1}^{\infty} \Big(\mathbb{E}_{\mu_0}^{\theta_0}\left[H\left(\{X_i\}_{i=k+tL}^{k+(t+1)L}\right)^2 \mathbbm{1}\left(X_{k+tL} \neq Y_{k+(t-1)L}\right)\right]^{\frac{1}{2}}\\
        &\quad +  \mathbb{E}_{\mu_0}^{\theta_0}\left[H\left(\{Y_i\}_{i=k+(t-1)L}^{k+tL}\right)^2 \mathbbm{1}\left(X_{k+tL} \neq Y_{k+(t-1)L}\right)\right]^{\frac{1}{2}}\Big)\\
        &\leq \left|\Gamma^{L}_{\kappa}\right|_{V}^{\frac{1}{2\kappa}}\mu_0\Bar{V}^{\frac{1}{2\kappa}} + 2\left|\Gamma^{L}_{\kappa}\right|_{V}^{\frac{1}{2\kappa}}\mu_0\Bar{V}^{\frac{1}{2\kappa}}\sum_{t=1}^{\infty} P_{x}\left(\tau^L > k + tL\right)^{\frac{\kappa - 1}{2\kappa}}\\
        &\leq \left|\Gamma^{L}_{\kappa}\right|_{V}^{\frac{1}{2\kappa}}\mu_0\Bar{V}^{\frac{1}{2\kappa}} + 2\left|\Gamma^{L}_{\kappa}\right|_{V}^{\frac{1}{2\kappa}}\mu_0\Bar{V}^{\frac{1}{2\kappa}}M^{\frac{\kappa - 1}{2\kappa}}\left(\mu_{0}V + \mu_{0}\Bar{V}\right)^{\frac{\kappa - 1}{2\kappa}}\left(\rho^{\frac{\kappa - 1}{2\kappa}}\right)^{k}\sum_{t=0}^{\infty}\left(\rho^{\frac{\kappa - 1}{2\kappa}}\right)^{tL},
    \end{aligned}
    \end{equation}
    where the first inequality follows from Minkowski's inequality; the second inequality follows from \eqref{eq: thm-bd-1} and the third inequality follows from \eqref{eq: thm-bd-2}.
    Also observe that the second term of the upper bound is actually an upper bound for $\mathbb{E}_{\mu_0}^{\theta_0}\left[\left(BC_{k}^L\right)^2\right]^{1/2}$.
    Since $\left(\rho^{\frac{\kappa - 1}{2\kappa}}\right)^{k}$ goes to $0$ as $k$ increases to $\infty$ and the other terms are constants, we may conclude \eqref{eq: thm-bc}:
    \begin{equation}
        \mathbb{E}_{\mu_0}^{\theta_0}\left[\left(BC_{k}^L\right)^2\right] \rightarrow 0
    \end{equation}
    as $k \rightarrow \infty$.
    Furthermore, see that 
    \begin{equation}
        \mathbb{E}_{\mu_0}^{\theta_0}\left[\left(SE_{k}^L\right)^2\right] = \mathbb{E}_{\mu_0}^{\theta_0}\left[H\left(\{X_i\}_{i=k}^{k+L}\right)^{2}\right] = \mathbb{E}_{\mu_0}^{\theta_0}\left[\Gamma^{L}_{\frac{1}{2}}(X_k)\right]
    \end{equation}
    and $\left|\Gamma^{L}_{1}\right|_{V} < \infty$ is implied by Lemma \ref{lemma: Gamma}. 
    Therefore,
    \begin{equation}
        \left|\mathbb{E}_{\mu_0}^{\theta_0}\left[\Gamma^{L}_{1}(X_k)\right] - \pi(\theta_0)\Gamma^{L}_{1}\right| \leq \left|\Gamma^{L}_{1}\right|_{V}M(\mu_0 V) \rho^k,
    \end{equation}
    where the upper bound goes to $0$ as $k$ increases to infinity.
    Hence we showed \eqref{eq: thm-se}.
    To prove $H_{2}^{k,L}(X, Y)$ is unbiased, first realize that 
    \begin{equation}
        \mathbb{E}_{\mu_0}^{\theta_0}\left[H\left(\{X_i\}_{i=k}^{k+L}\right)\right] = \mathbb{E}_{\mu_0}^{\theta_0}\left[h(X_k)\right].
    \end{equation}
    Then using the fact that $H_{2}^{k,L}(X, Y)$ has finite second moment, we can guarantee that it has finite first moment and thus we may invoke the Fubini's theorem to guarantee the validity of the interchange of the summation and expectation from \eqref{eq: begin-g} to \eqref{eq: end-g} and conclude that $H_{2}^{k,L}(X, Y)$ is indeed an unbiased estimator. 
    The finite computation time is implied by the geometrically decaying tail probability of the coupling time made in Assumption \ref{assumption-coupling}.
\end{proof}

\subsection{Proof of Lemma \ref{lemma: efficiency}}\label{sec: proof-lemma3}
\begin{proof}
    First we bound the variance $\sigma_L^2$ by its second moment:
    \begin{equation}\label{eq: theorem2-bound1}
        \sigma_{L}^2 \leq \pi(\theta_0)\Gamma_1^L \leq \pi(\theta_0)V |\Gamma_1^L|_{V}
    \end{equation}
    Now we turn to obtain an upper bound for the covariance term. 
    To ease the notations, we let
    \begin{equation}
        \newnota{HLj}{H^L(j)} := H\left(\{X_i\}_{i=j}^{j + L}\right).
    \end{equation}
    Observe that
    \begin{equation}\label{eq: H0HjL}
    \begin{aligned}
        \mathbb{E}_{\pi(\theta_0)}^{\theta_0}\left[H^L(0) H^L(jL)\right] &= \mathbb{E}_{\pi(\theta_0)}^{\theta_0}\left[H^L(0)\mathbb{E}_{X_L}^{\theta_0}\left[H^L(jL)\right]\right]\\
        &=\mathbb{E}_{\pi(\theta_0)}^{\theta_0}\left[H^L(0)\left(\gamma'(\theta_0) + \mathbb{E}_{X_L}^{\theta_0}\left[H^L(jL)\right] - \gamma'(\theta_0)\right)\right]\\
        &= \gamma'(\theta_0)^2 + \mathbb{E}_{\pi(\theta_0)}^{\theta_0}\left[H^L(0)\left(\mathbb{E}_{X_L}^{\theta_0}\left[H^L(jL)\right] - \gamma'(\theta_0)\right)\right]\\
        &\leq \gamma'(\theta_0)^2 + \mathbb{E}_{\pi(\theta_0)}^{\theta_0}\left[H^L(0)^2\right]^{\frac{1}{2}}\mathbb{E}_{\pi(\theta_0)}^{\theta_0}\left[\left(\mathbb{E}_{X_L}^{\theta_0}\left[H^L(jL)\right] - \gamma'(\theta_0)\right)^{2}\right]^{\frac{1}{2}}.
    \end{aligned}
    \end{equation}
    Now note that since $j \geq 1$,
    \begin{align}
        \mathbb{E}_{x}^{\theta_0}\left[H^L(jL)\right] &= \mathbb{E}_{x}^{\theta_0}\left[\Gamma^{L}_{\frac{1}{2}}\left(X_{(j-1)L}\right)\right],\\
        \gamma'(\theta_0) &= \pi(\theta_0)\Gamma^{L}_{\frac{1}{2}},
    \end{align}
    and by Jensen's inequality,
    \begin{equation}
       |\Gamma^{L}_{\frac{1}{2}}(x)|^2 \leq |\Gamma^{L}_{1}(x)| \leq \left|\Gamma^{L}_{1}\right|V(x).
    \end{equation}
    Thus
    \begin{equation}\label{eq: H0HjL-2}
        \left|\mathbb{E}_{x}^{\theta_0}\left[H^L(jL)\right] - \gamma'(\theta_0)\right| \leq \left|\Gamma^{L}_{1}\right|^{\frac{1}{2}}_{V}M\sqrt{V(x)}\rho^{(j-1)L}
    \end{equation}
    Therefore,
    \begin{align}
        \sigma_{j, L} &= \mathbb{E}_{\pi(\theta_0)}^{\theta_0}\left[H^L(0) H^L(jL)\right] - \gamma'(\theta_0)^2\\
        &\leq \mathbb{E}_{\pi(\theta_0)}^{\theta_0}\left[H^L(0)^2\right]^{\frac{1}{2}}\mathbb{E}_{\pi(\theta_0)}^{\theta_0}\left[\left(\mathbb{E}_{X_L}^{\theta_0}\left[H^L(jL)\right] - \gamma'(\theta_0)\right)^{2}\right]^{\frac{1}{2}}\\
        &\leq \left(\pi(\theta_0)\Gamma^{L}_{1}\right)^{\frac{1}{2}} \left(\pi(\theta_0)V\right)^{\frac{1}{2}}\left|\Gamma^{L}_{1}\right|^{\frac{1}{2}}_{V}M\rho^{(j-1)L}\\
        &\leq \left(\pi(\theta_0)V\right)\left|\Gamma^{L}_{1}\right|_{V}M\rho^{(j-1)L}
    \end{align}
    where the first inequality follows from \eqref{eq: H0HjL} and the second inequality follows from realizing that $\mathbb{E}_{\pi(\theta_0)}^{\theta_0}\left[H^L(0)^2\right] = \pi(\theta_0)\Gamma^{L}_{1}$ and inequality \eqref{eq: H0HjL-2}. 
    Thus we can derive an upper bound for the sum of the covariance terms 
    \begin{equation}\label{eq: theorem2-bound2}
        \begin{aligned}
            \sum_{j=1}^{\infty} \sigma_{j, L} &\leq \sigma_{1, L} + \sum_{j=2}^{\infty} \sigma_{j, L} \leq \sigma_L^2 + \left(\pi(\theta_0)V\right)\left|\Gamma^{L}_{1}\right|_{V}\frac{M\rho^L}{1-\rho^L}\\
        \end{aligned}
    \end{equation}
    Finally we combine the upper bounds for variance \eqref{eq: theorem2-bound1} and covariance terms \eqref{eq: theorem2-bound2} and invoke Lemma \ref{lemma: Gamma} to get the upper bound for the asymptotic variance:
    \begin{equation}
        \sigma_L^2 + 2\sum_{j=1}^{\infty} \sigma_{j, L} \leq U_{1, z}(L)\pi(\theta_0)V\left(3 + \frac{2M\rho^L}{1-\rho^L}\right)
    \end{equation}
\end{proof}

\subsection{Proof of Theorem \ref{theorem: efficiency}}\label{sec: proof-theorem4}
\begin{proof}
    Note that
    \begin{equation}
        \frac{\frac{2M\rho^L}{1-\rho^L}}{\frac{2M\rho}{1-\rho}} \leq 1,
    \end{equation}
    for all $L \geq 1$. 
    Thus
    \begin{equation}
        \frac{W(1)}{W(L)} \geq \frac{\left(1 + 2\mathbb{E}[\tau_{\pi(\theta_0), z}]\right)U_{1, z}(1)}{\left(L + 2\mathbb{E}[\tau_{\pi(\theta_0), z}]\right)U_{1, z}(L)},
    \end{equation}
    and our goal is to develop a lower bound for the right hand side. 
    Before doing so, we make obvious some useful inequalities. 
    First see that
    \begin{equation}
    \begin{aligned}
    \mathbb{E}\left[\tau_{\pi(\theta_0),z}\right] - 1 &\leq \sum_{t=1}^\infty P\left(\tau_{\pi(\theta_0), z} > t\right)\\
    &\leq \sum_{t=1}^\infty \left(P\left(\tau_{\pi(\theta_0), z} > t\right)\right)^{\frac{\delta}{2 p(2 p+\delta)}}\\
    &\leq M^{\frac{\delta}{2 p(2 p+\delta)}} \left(\pi(\theta_0)V + V(z)\right)^{\frac{\delta}{2 p(2 p+\delta)}}\frac{\rho^{\frac{\delta}{2 p(2 p+\delta)}}}{1-\rho^{\frac{\delta}{2 p(2 p+\delta)}}},
    \end{aligned}
    \end{equation}
    where the last inequality follows from \eqref{eq: Etau}. 
    Rearrange the inequality to get
    \begin{equation}
        \begin{aligned}
            \left(\frac{\rho^{\frac{\delta}{2 p(2 p+\delta)}}}{1-\rho^{\frac{\delta}{2 p(2 p+\delta)}}}\right)^{-2} &\leq \left(\mathbb{E}\left[\tau_{\pi(\theta_0),z}\right] - 1\right)^{-2}M^{\frac{2\delta}{2 p(2 p+\delta)}} \left(\pi(\theta_0)V + V(z)\right)^{\frac{2\delta}{2 p(2 p+\delta)}}\\
            &\leq \left(\mathbb{E}\left[\tau_{\pi(\theta_0),z}\right] - 1\right)^{-2}M^{\frac{2\delta}{2 p(2 p+\delta)}} \left(\pi(\theta_0)V^{\frac{2\delta}{2 p(2 p+\delta)}} + V(z)^{\frac{2\delta}{2 p(2 p+\delta)}}\right)\\
            &\leq \left(\mathbb{E}\left[\tau_{\pi(\theta_0),z}\right] - 1\right)^{-2}M^{\frac{2\delta}{2 p(2 p+\delta)}} \left(\left(\frac{b}{1-\lambda}\right)^{\frac{2\delta}{2 p(2 p+\delta)}} + V(z)^{\frac{2\delta}{2 p(2 p+\delta)}}\right),
        \end{aligned}
    \end{equation}
    where the last inequality follows from that for any $x \in \mathcal{X}$:
    \begin{equation}
        \pi(\theta_0)V = \lim_{n \rightarrow \infty} \mathbb{E}_x^{\theta_0}\left[V(X_n)\right] \leq \lim_{n \rightarrow \infty} \lambda^n V(x) + \frac{b}{1-\lambda} = \frac{b}{1-\lambda}.
    \end{equation}
    Now recall that
    \begin{equation}
        \frac{A_z^{1}}{B_z^{1}} = M^{-\frac{2\delta}{2 p(2 p+\delta)}}\left(1 + \left(\frac{b}{1-\lambda}\right)^{\frac{2\delta}{2 p(2 p+\delta)}}+V(z)^{\frac{2\delta}{2 p(2 p+\delta)}}\right)^{-1},
    \end{equation}
    and we now have the following inequality:
    \begin{equation}\label{eq: theorem-ineq}
        \frac{A_z^{1}}{B_z^{1}}\left(\frac{\rho^{\frac{\delta}{2 p(2 p+\delta)}}}{1-\rho^{\frac{\delta}{2 p(2 p+\delta)}}}\right)^{-2} \leq \left(\mathbb{E}\left[\tau_{\pi(\theta_0),z}\right] - 1\right)^{-2}.
    \end{equation}
    Now we go back to developing the lower bound for the performance ratio:
    \begin{equation}
    \begin{aligned}
        \frac{\left(1 + 2\mathbb{E}[\tau_{\pi(\theta_0), z}]\right)U_{1, z}(1)}{\left(L + 2\mathbb{E}[\tau_{\pi(\theta_0), z}]\right)U_{1, z}(L)} &= \frac{\left(2\mathbb{E}\left[\tau_{\pi(\theta_0),z}\right] + 1\right)\left(A_z^{1} + \left(\frac{\rho^{\frac{\delta}{2p(2p+\delta)}}}{1-\rho^{\frac{\delta}{2p(2p+\delta)}}}\right)^2B_z^{1}\right)}{\left(2\mathbb{E}\left[\tau_{\pi(\theta_0),z}\right] + L\right)\left(LA_z^{1} + L\left(\frac{\left(\rho^{\frac{\delta}{2p(2p+\delta)}}\right)^L}{1-\left(\rho^{\frac{\delta}{2p(2p+\delta)}}\right)^L}\right)^2B_z^{1}\right)}\\
        &\geq  \frac{2\mathbb{E}\left[\tau_{\pi(\theta_0),z}\right]\left(\frac{\rho^{\frac{\delta}{2p(2p+\delta)}}}{1-\rho^{\frac{\delta}{2p(2p+\delta)}}}\right)^2B_z^{1}}{\left(2\mathbb{E}\left[\tau_{\pi(\theta_0),z}\right] + L\right)\left(LA_z^{1} + \frac{1}{L}\left(\frac{\rho^{\frac{\delta}{2p(2p+\delta)}}}{1-\rho^{\frac{\delta}{2p(2p+\delta)}}}\right)^2B_z^{1}\right)}\\
        &= \frac{1}{\left(1 + \frac{L}{2\mathbb{E}\left[\tau_{\pi(\theta_0),z}\right]}\right)\left(L\left(\frac{\rho^{\frac{\delta}{2p(2p+\delta)}}}{1-\rho^{\frac{\delta}{2p(2p+\delta)}}}\right)^{-2} \frac{A_z^{1}}{B_z^{1}} + \frac{1}{L}\right)}\\
        &\geq \frac{1}{\left(1 + \frac{L}{2\mathbb{E}\left[\tau_{\pi(\theta_0),z}\right]}\right)\left(L\frac{1}{\left(\mathbb{E}\left[\tau_{\pi(\theta_0),z}\right] - 1\right)^2} + \frac{1}{L}\right)},
    \end{aligned}
    \end{equation}
    where the first inequality follows from the fact that $\frac{h^x}{1-h^x} <= \frac{1}{x}\frac{h}{1-h}$ for any $0 < h < 1$ and $x \geq 1$ and the last inequality follows from \eqref{eq: theorem-ineq}.
    Finally if we let $L = \mathbb{E}\left[\tau_{\pi(\theta_0),z}\right] - 1$, then we observe
    \begin{equation}
        \frac{W(1)}{W(\mathbb{E}\left[\tau_{\pi(\theta_0),z}\right]-1)} > \frac{1}{3}\mathbb{E}\left[\tau_{\pi(\theta_0),z}-1\right].
    \end{equation}
    which completes the proof.
\end{proof}

\section{Tables}
\begin{table}[h]
\centering
\begin{tabular}{|c||c||c||c|}
\hline
     N &  Work &  Work$\cdot$Variance ($10^6$) &  Estimate (True Value:-24.963)\\ 
\hline
    100 & $5,235\pm24$ & $66.90\pm9.88$ & $-37.14\pm1.07$\\ 
\hline
    200 & $10,404\pm28$ & $47.74\pm9.43$ & $-30.90\pm0.53$\\ 
\hline
    500 & $26,151\pm160$ & $33.60\pm7.46$ & $-27.28\pm0.99$ \\ 
\hline
    1000 & $51,943\pm77$ & $30.15\pm1.58$ & $-26.18\pm0.23$ \\ 
\hline
    2000 & $104,019\pm150$ & $27.54\pm1.80 $ & $-25.47\pm0.21$  \\ 
\hline
    5000 & $259,609\pm395$ & $25.98\pm1.42$ &  $-25.00\pm0.21$ \\ 
\hline
    10000 & $520,093\pm457$ & $25.27\pm0.92$ & $-24.97\pm0.07$  \\ 
\hline
    40000 & $2,080,745\pm1147$ & $26.20\pm0.83$ & $-25.10\pm0.09$   \\ 
\hline
\end{tabular}
\caption{Performance stats of $H^N_{IPA}$ with different $N$}
\label{tab:performance}
\end{table}

\begin{table}[h]
\centering
\begin{tabular}{|c||c||c||c|}
\hline
     N &  Work &  Work$\cdot$Variance ($10^6$) &  Estimate (True Value:-24.963)\\ 
\hline
    100 & $186071\pm4646$ & $38.17\pm2.37$ & $-18.16\pm0.30$\\ 
\hline
    200 & $382259\pm6929$ & $63.03\pm4.85$ & $-21.15\pm0.29$\\ 
\hline
    500 & $943558\pm10339$ & $88.74\pm5.23$ & $-22.92\pm0.20$ \\ 
\hline
    1000 & $1896955\pm13497$ & $112.83\pm5.90$ & $-23.97\pm0.15$ \\ 
\hline
    2000 & $3755492\pm27625$ & $115.11\pm6.87 $ & $-24.38\pm0.16$  \\ 
\hline
    3000 & $5643306\pm28646$ & $132.32\pm8.47$ & $-24.56\pm0.15$ \\ 
\hline
\end{tabular}
\caption{Performance stats of $H^N_{Ph}$ with different $N$}
\label{tab:performance-ph}
\end{table}

\begin{table}[h]
\centering
\begin{tabular}{|c||c||c||c|}
\hline
     N &  Work ($10^4$)&  Work $\cdot$Variance ($10^6$) &  Estimate (True Value:-24.963)\\ 
\hline
    1000 & $2.59\pm0.001$ & $15.33\pm0.36$ & $-18.40\pm0.07$\\ 
\hline
    2000 & $5.19\pm0.006$ & $28.60\pm2.25$ & $-21.28\pm0.18$\\ 
\hline
    3000 & $7.80\pm0.005$ & $34.86\pm1.50$ & $-22.38\pm0.11$ \\ 
\hline
    5000 & $12.97\pm0.03$ & $39.42\pm2.01$ & $-22.75\pm0.42$ \\ 
\hline
    10000 & $25.99\pm0.02$ & $45.23\pm1.55$ & $-24.02\pm0.08$  \\ 
\hline
    20000 & $52.02\pm0.03$ & $51.37\pm2.88$ &  $-24.54\pm0.14$ \\ 
\hline
    40000 & $104.00\pm0.05$  & $57.30\pm3.66$ & $-24.86\pm0.04$  \\ 
\hline
\end{tabular}
\caption{Performance stats of $H^N_{Re}$ with different $N$}
\label{tab:performance-lr}
\end{table}

\begin{table}[h]
\centering
\begin{tabular}{|c||c||c||c|}
\hline
     N  &  Work$\cdot$Variance ($10^6$) &  Estimate (True Value:-24.963)\\ 
\hline
    2000  & $0.60\pm0.009$ & $-10.19\pm0.03$\\ 
\hline
    5000 & $4.64\pm0.07$ & $-16.33\pm0.05$ \\ 
\hline
    10000  &$14.41\pm0.32$ & $-20.17\pm0.08$\\ 
\hline
    20000 & $31.61\pm1.28$ & $-22.57\pm0.12$ \\ 
\hline
    50000 & $64.02\pm3.32$ & $-23.98\pm0.21$  \\ 
\hline
    100000 & $114.23\pm3.90$ &  $-24.62\pm0.18$ \\ 
\hline
    500000   & $479.88\pm28.82$ & $-25.26\pm0.38$  \\ 
\hline
\end{tabular}
\caption{Performance stats of $H^N_{LR}$ with different $N$}
\label{tab:performance-LR}
\end{table}

\begin{table}[h]
\centering
\begin{tabular}{|c||c||c||c|}
\hline
     L &  Work ($10^4$)&  Work $\cdot$Variance ($10^6$) &  Estimate (True Value:-24.963)\\ 
\hline
    30 & $5.09\pm0.01$ & $7957.25\pm3905.91$ & $-25.11\pm1.57$\\ 
\hline
    100 & $6.74\pm0.01$ & $412.25\pm201.43$ & $-24.77\pm0.31$\\ 
\hline
    200 & $8.07\pm0.02$ & $110.36\pm39.80$ & $-25.06\pm0.22$ \\ 
\hline
    500 & $11.27\pm0.05$ & $36.40\pm4.98$ & $-24.85\pm0.37$ \\ 
\hline
    1000 & $16.46\pm0.06$ & $32.15\pm3.59$ & $-25.12\pm0.45$  \\ 
\hline
    1500 & $21.46\pm0.05$ & $32.91\pm2.98$ &  $-24.65\pm0.38$ \\ 
\hline
    3000 & $36.65\pm0.05$  & $47.34\pm3.66$ & $-24.89\pm0.38$  \\ 
\hline
    10000 & $107.50\pm0.03$  & $215.92\pm7.80$ & $-24.85\pm0.35$  \\ 
\hline
\end{tabular}
\caption{Performance stats of $H_{4}^{L, 100, L}$ with different $N$}
\label{tab:performance-h4}
\end{table}

\mdefnota{chain-X}{$\{X_i\}_{i\geq 0}\text{ : a Markov chain on state space} \mathcal{X}$}
\mdefnota{state-space}{$\text{General state space for the Markov chains}$}
\mdefnota{P-theta}{$\left(P(\theta, x, dy): x, y \in \mathcal{X}\right)\text{is the one-step transition kernel of chain X parametrized by } \theta$}
\mdefnota{pi-theta}{Invariant distribution associated with $P(\theta)$}
\mdefnota{gamma}{$\mathbb{E}_{\pi(\theta)}[f(X)]$: The stationary mean}
\mdefnota{f}{A function that maps from $\mathcal{X}$ to $\mathbb{R}$ and is integrable under $\pi(\theta)$}
\mdefnota{gamma-prime}{$\frac{d}{d\theta}\mathbb{E}_{\pi(\theta)}[f(X)]$: The derivative of the stationary mean w.r.t. $\theta$}
\mdefnota{P+}{The positive part of the Hahn-Jordan decomposition of $P'(\theta)$}
\mdefnota{P-}{The negative part of the Hahn-Jordan decomposition of $P'(\theta)$}
\mdefnota{P-bar}{The joint transition kernel for a couple of Markov chains}
\mdefnota{mu-0}{Initial distribution for the Markov chain}
\mdefnota{tau^L}{$\inf\left(i \geq L: X_{i} = Y_{i-L}\right)$: The meeting time of the couple chain}
\mdefnota{couple-xy}{$\{X_i, Y_i\}_{i\geq 0}$, the coupled Markov chain on $\mathcal{X} \times \mathcal{X}$}
\mdefnota{k}{$k \geq 0$: burn-in period}
\mdefnota{m}{$m \geq k$: last index of the average in $H^L_{k:m}$}
\mdefnota{Hk}{$f(X_k)  + \sum_{t=1}^{\floor*{\frac{\tau^L-k}{L}}} f(X_{k + tL}) - f(Y_{k+(t-1)L})$: an unbiased estimator for $\gamma(\theta)$}
\mdefnota{Hk-1}{$H(X_k, X_{k+1})  + \sum_{t=1}^{\floor*{\frac{\tau^L-k}{L}}} H(X_{k + tL}, X_{k + tL + 1}) - H(Y_{k+(t-1)L}, Y_{k+(t-1)L + 1})$}
\mdefnota{Hk-2}{$H\left(\{X_i\}_{i=k}^{k+L}\right)  + \sum_{t=1}^{\floor*{\frac{\tau^L-k}{L}}} H\left(\{X_i\}_{i=k+tL}^{k+(t+1)L}\right) - H\left(\{Y_i\}_{i=k+(t-1)L}^{k+tL}\right)$}
\mdefnota{H-km}{$\frac{1}{m-k+1}\sum_{t=k}^{m}H_t^L$}
\mdefnota{L}{$L \geq 1$: the lagging parameter}
\mdefnota{ct}{$\floor*{\frac{\min(m+L, t) - k - (t - k)\% L}{L}}$: the coefficients in the bias correction term of $H^L_{k:m}$}
\mdefnota{g}{Solution to the Poisson's equation $g(x) - P(\theta)g(x) = f(x) - \gamma(\theta)$}
\mdefnota{Ph}{$\int_{\mathcal{X}}h(x)P(x, dy)$}
\mdefnota{mu-h}{$\int_{\mathcal{X}}h(x)\mu(dy)$}
\mdefnota{g-regen}{$\mathbb{E}_x\left[\sum_{t=0}^{T(\alpha)}\left(f(X_t) - \pi(\theta)f\right)\right]$: regeneration solution to the Poisson's equation}
\mdefnota{g-funda}{$\mathbb{E}_x^{\theta}\left[\sum_{t=0}^{\infty}\left(f(X_t) - \pi(\theta)f\right)\right]$: fundamental solution to the Poisson's equation}
\mdefnota{alp}{A regeneration state for the Markov chain}
\mdefnota{T-alp}{$\inf\{t > 0: X_t = \alpha\}$ The regeneration time for the Markov chain}
\mdefnota{tau-xz}{$\inf\{t \geq 0: X_t = Z_t\}$: The coupling time of the joint chain $(X,Z)$ that starts from $(x, z)$}
\mdefnota{Lam}{Space for parameter $\theta$}
\mdefnota{theta-0}{$\theta_0$: Current parameter at which we are trying to compute the derivative}
\mdefnota{G-z}{$\sum_{t=0}^{\tau_{x,z}}\left(f(X_t) - f(Z_t)\right)$: an unbiased estimator for $g_z(x)$}
\mdefnota{G-z-L}{$\sum_{i=0}^{\floor*{\frac{\tau_{x, z}}{L}}}f(X_{iL}) - f(Z_{iL})$: Unbiased estimator for the solution to the Poisson's equation associated with the $L$-skeleton chain}
\mdefnota{g-z}{$\mathbb{E}_{x, z}^{\theta}\left[\sum_{t=0}^{\tau_{x,z}}\left(f(X_t) - f(Z_t)\right)\right]$: solution to the Poisson's equation}
\mdefnota{g-z-L}{$\mathbb{E}\left[\sum_{i=0}^\infty f(X_{iL}) - f(Z_{iL})\right]$: solution to the Poisson's equation associated to the $L$-skeleton chain}
\mdefnota{h-x}{$\mathbb{E}_{x}^\theta\left[p'(\theta, x, X_1)g(X_1)\right]$}
\mdefnota{H-x0x1}{$p'(\theta_0, x_0, x_1)G_z(x_1)$}
\mdefnota{P-L}{$L$-step transition kernel}
\mdefnota{g-L}{Solution to $g^{L} - P^{L}(\theta)g^{L} = f - \pi(\theta)f$}
\mdefnota{H-0-L}{$\left(\sum_{i=0}^{L-1}p'(\theta_0, x_{i}, x_{i+1})\right)g^L(x_L)$}
\mdefnota{M-L}{$\sum_{i=0}^{L-1}p'(\theta_0, x_{i}, x_{i+1})$}
\mdefnota{K}{A small set}
\mdefnota{Gamma-x}{$\mathbb{E}_x^{\theta_0}\left[H\left(\{X_i\}_{i=0}^{L}\right)^{2\zeta}\right]$}
\mdefnota{omega-epsilon}{$\sup _{\left|\theta-\theta_0\right|<\epsilon}\left|p^{\prime}(\theta, x, y)\right|$}
\mdefnota{Bar-V}{$V(x) + \frac{b}{1-\lambda}$}
\mdefnota{C-kappa}{$\left[8(l-1) \max(1, 2^{l-3})\right]^{l}$}
\mdefnota{Omega}{$\int_{\mathcal{X}}\omega_{\epsilon}(x, y)^{\frac{2\kappa p}{p-1}}P(\theta_0, x, dy)$}
\mdefnota{Omega-2}{$\int_{\mathcal{X}}\omega_{\epsilon}(x, y)P(\theta_0, x, dy)$}
\mdefnota{f2p}{$f(x)^{2\kappa p+\delta}$}
\mdefnota{p}{$p > 1$}
\mdefnota{kappa}{$\kappa > 1$}
\mdefnota{delta}{$\delta > 0$}
\mdefnota{C-Gamma}{$C_\Gamma^{L} > 0$}
\mdefnota{D-Gamma}{$D_\Gamma^{L} > 0$}
\mdefnota{V}{Lyapunov function such that $P(\theta_0) V(x) \leq \lambda V(x)+b 1_K(x)$ holds}
\mdefnota{Bar-V}{$V(x) + \frac{b}{1-\lambda}$}
\mdefnota{SE-kL}{$H\left(\{X_i\}_{i=k}^{k+L}\right)$: single sample gradient estimator}
\mdefnota{BC-kL}{$\sum_{t=1}^{\floor*{\frac{\tau^L-k}{L}}} H\left(\{X_i\}_{i=k+tL}^{k+(t+1)L}\right) - H\left(\{Y_i\}_{i=k+(t-1)L}^{k+tL}\right)$: bias correction term to the singleton estimator}
\mdefnota{H-3}{$\frac{1}{m-k+1}\sum_{t=k}^{m}H_2^{t, L}(X, Y)$}
\mdefnota{H-4}{$\frac{1}{m}\sum_{t=0}^{m-1}H_2^{k+tL, L}(X, Y)$}
\mdefnota{sig-2}{$Var_{\pi(\theta_0)}\left(H\left(\{X_i\}_{i=0}^{L}\right)\right)$}
\mdefnota{rho-jL}{$Cov_{\pi(\theta_0)}\left(H\left(\{X_i\}_{i=0}^{L}\right), H\left(\{X_i\}_{i=jL}^{(j+1)L}\right)\right)$}
\mdefnota{tau-pi}{The coupling time of two chains starting from the stationary distribution $\pi(\theta_0)$ and a pre-specified state $z$}
\mdefnota{R}{$R < \infty$: Constant}
\mdefnota{r}{$0 < r < 1$: Constant}
\mdefnota{M}{$R < M < \infty$: Constant}
\mdefnota{rho}{$r < \rho < 1$: Constant}
\mdefnota{HLj}{$H\left(\{X_i\}_{i=j}^{j + L}\right)$}
\mdefnota{U-zeta}{$\left(L^{\zeta}A_{z} + L^{\zeta}\left(\frac{\left(\rho^{\frac{\delta}{2\zeta p(2\zeta p+\delta)}}\right)^L}{1-\left(\rho^{\frac{\delta}{2\zeta p(2\zeta p+\delta)}}\right)^L}\right)^{2\zeta}B_{z}\right)$}
\mdefnota{A_z}{$2^{4\zeta - 2}\left|f^{2\zeta p+\delta}\right|_{V}^{\frac{2\zeta}{2\zeta p+\delta}}\left(C_{\frac{2\zeta p}{p-1}}|\Omega^{\frac{2\zeta p}{p-1}}|_{V} \left(1 + \frac{b}{1-\lambda}\right)\right)^{\frac{p-1}{ p}}\left(1 +\left(\frac{b}{1-\lambda}\right)^{\frac{2\zeta}{2\zeta p+\delta}}+\Bar{V}(z)^{\frac{2\zeta}{2\zeta p+\delta}}\right)$}
\mdefnota{B_z}{$A_z^{\zeta} \cdot M^{\frac{2\zeta\delta}{2\zeta p(2\zeta p+\delta)}}\left(1 + \left(\frac{b}{1-\lambda}\right)^{\frac{2\zeta\delta}{2\zeta p(2\zeta p+\delta)}}+V(z)^{\frac{2\zeta\delta}{2\zeta p(2\zeta p+\delta)}}\right)$}
\mdefnota{W-L}{$\left(L + 2\mathbb{E}[\tau_{\pi(\theta_0), z}]\right)U_{1, z}(L)\pi(\theta_0)V\left(3 + \frac{2M\rho^L}{1-\rho^L}\right)$}
\mdefnota{S}{Exponential random variable with rate $\theta$}
\mdefnota{T}{Exponential random variable with rate $5$}
\mdefnota{g-fu-L}{$\mathbb{E}_{x}^{\theta_0}\left[\sum_{i}^{\infty}\bar{f}(X_{iL})\right]$}
\mdefnota{bar-f}{$f(x) - \pi(\theta_0)f$}

\bibliographystyle{plain}
\bibliography{refs.bib}

\ifshowtheoremlist
\newpage
\footnotesize
\linkdest{location of theorem list}
\footnotesize
\renewcommand{\listtheoremname}{List of Definitions and Notations}
\listoftheorems[ignoreall, show={definition}]
\renewcommand{\listtheoremname}{List of Theorems}
\listoftheorems[ignoreall, show={theorem,lemma,corollary,proposition,result}]
\renewcommand{\listtheoremname}{List of Assumptions}
\listoftheorems[ignoreall, show={assumption,condition}]
\renewcommand{\listtheoremname}{List of Remarks}
\listoftheorems[ignoreall, show={remark}]
\normalsize
\fi

\ifshownotationindex
\newpage
\linkdest{location of notation index}
\notationindex
\fi

\ifshowequationlist
\newpage
\section*{List of Numbered Equations}
\linkdest{location of equation number list}
\fi

\ifshownavigationpage
\newpage
\normalsize

\section*{Navigation Links}

\ifshowequationlist
    \noindent
    \hyperlink{location of equation number list}{Numbered Equations}
    \bigskip
\fi

\ifshowtheoremlist
    \noindent 
    \hyperlink{location of theorem list}{Theorem List}
    \bigskip
\fi

\ifshownotationindex
    \noindent
    \hyperlink{location of notation index}{Notation Index}
    
\fi

\ifshowtoc
    \newpage
    \setcounter{tocdepth}{2}
    \tableofcontents
\fi

\fi

\end{document}